\documentclass[11pt, twoside,leqno]{article}

\usepackage[english]{babel}

\usepackage{amssymb}
\usepackage{amsfonts}
\usepackage{amsmath}
\usepackage{amsthm}
\usepackage{color}
\usepackage{mathrsfs}
\usepackage{float}
\usepackage{caption}
\usepackage{tikz}
\usetikzlibrary{arrows,decorations.pathreplacing}
\definecolor{myred}{RGB}{255,0,0}
\definecolor{myblue}{RGB}{0,0,255}
\usepackage{txfonts}
\usepackage{enumerate}
\usepackage{anysize}
\usepackage{indentfirst}

\usepackage{latexsym}

\usepackage[colorlinks=true,
linkcolor=blue,
citecolor=red,
urlcolor=magenta,
]{hyperref}

\allowdisplaybreaks

\newtheorem{theorem}{Theorem}[section]
\newtheorem{lemma}[theorem]{Lemma}
\newtheorem{corollary}[theorem]{Corollary}
\newtheorem{proposition}[theorem]{Proposition}

\theoremstyle{definition}
\newtheorem{remark}[theorem]{Remark}
\newtheorem{definition}[theorem]{Definition}

\newcounter{assum}

\numberwithin{equation}{section}

\begin{document}

\arraycolsep=1pt

\title{\bf\Large Sharp Quantitative Matrix-Weighted Bounds for Fractional Integrals and Sobolev Inequalities\footnotetext{ \hspace{-0.35cm} 2020 \emph{Mathematics Subject Classification}.
Primary 47G40; Secondary 46E35, 42B20, 47A30, 46E40, 42B35.
\endgraf \emph{Key words and phrases}. matrix weight,
fractional integral, Sobolev inequality.}}
\author{Dachun Yang, Wen Yuan and Mingdong Zhang}
\date{\today}
\maketitle

\vspace{-0.6cm}

\begin{center}
\begin{minipage}{13cm}
{\small {\bf Abstract}\quad
For $\alpha\in(0,n)$, $p\in(1,\frac{n}{\alpha})$,
$q:=\frac{np}{n-\alpha p}$, and $W\in\mathscr A_{p,q}$,
we obtain the following matrix-weighted boundedness of
the fractional integral $I_\alpha$: for any $\vec{f}\in L^p(W^p)$,
\begin{align*}
\left\|I_\alpha\vec f\right\|_{L^q(W^q)}
\lesssim[W]_{\mathscr A_{p,q}}^{
(1-\frac{\alpha}{n})\max\{1,\frac{p'}{q}\}}
\left\|\vec f\right\|_{L^p(W^p)},
\end{align*}
where the implicit positive constant is
independent of $W$ and $\vec{f}$.
For integer $n\geq2$, $p\in[1,n)$, $q:=\frac{np}{n-p}$,
and $W\in\mathscr A_{p,q}$, we establish the following matrix-weighted
Sobolev inequality: for any smooth $\mathbb{C}^d$-valued function
$\vec{f}$ with compact support,
\begin{align*}
\left\|\vec{f}\right\|_{L^{q}(W^q)}
\lesssim[W]_{\mathscr A_{p,q}}^{\frac{n-1}{n}}
\left\|WD\vec{f}\right\|_{L^p(\mathbb R^n,\mathbb C^{d\times n})},
\end{align*}
where the implicit positive constant is
independent of $W$ and $\vec{f}$ and where $D\vec{f}$
is the Jacobian matrix of $\vec{f}$.
Surprisingly, in both bounds above the exponents of $[W]_{\mathscr A_{p,q}}$
coincide with the corresponding known optimal scalar exponents,
and hence are optimal, which is different from the
Calder\'on--Zygmund operator case. A key idea to
obtain the above matrix-weighted boundedness of
the fractional integrals is to establish the
off-diagonal chain-packing principle of dyadic cubes,
which is inspired by a recent work of A. K. Lerner.
}
\end{minipage}
\end{center}




\section{Introduction}

Weighted norm inequalities for classical operators play a key role
in harmonic analysis. We first present a brief history of
this topic. Recall that a non-negative locally integrable
function on $\mathbb{R}^n$ is called a \emph{weight} if it
takes values in $(0, \infty)$ almost everywhere
(see, for instance, \cite[p.\,499]{g14c}).
For any weight $w$ and any $p\in[1,\infty)$, the \emph{weighted
Lebesgue space} $L^p(w)$ is defined to
be the set of all measurable functions
$f$ on $\mathbb{R}^n$ such that
\begin{align}\label{eq-Lp}
\|f\|_{L^p(w)}:=\left[\int_{\mathbb R^n}
|f(x)|^p w(x)\,dx\right]^{\frac1p}<\infty.
\end{align}
For any $p\in(1,\infty)$ and any weight $w$,
Muckenhoupt \cite{m72} proved that the Hardy--Littlewood maximal operator $M$ is
bounded on $L^p(w)$ if and only if $w$ satisfies the $A_p$ condition:
\begin{align}\label{eq-Ap}
[w]_{A_p}:=\sup_{Q\subset\mathbb{R}^n}\fint_Q w(x)\,dx
\left(\fint_Q [w(x)]^{-\frac{1}{p-1}}\,dx\right)^{p-1}
<\infty,
\end{align}
where the supremum is taken over all cubes $Q\subset\mathbb R^n$
(all cubes $Q\subset\mathbb{R}^n$ in this article are
always assumed to have edges parallel to the coordinate axes).
Subsequently Hunt et al. \cite{hmw73} proved that,
for any $p\in(1,\infty)$, the $A_p$ condition \eqref{eq-Ap}
also characterizes the boundedness of the Hilbert
transform on $L^p(w)$. Recall that, for any
$\alpha\in(0,n)$, the \emph{fractional integral operator} $I_\alpha$
is defined by setting, for any suitable measurable function $f$
on $\mathbb{R}^n$ and for any $x\in\mathbb{R}^n$,
\begin{align*}
I_\alpha f(x):=\int_{\mathbb R^n}
\frac{f(y)}{|x-y|^{n-\alpha}}\,dy.
\end{align*}
Let
\begin{align}\label{eq:range}
\alpha\in(0,n),\ p\in\left(1,\frac{n}{\alpha}\right),
\text{ and }q:=\frac{np}{n-\alpha p}.
\end{align}
In this off-diagonal setting, for any weight $w$,
Muckenhoupt and Wheeden \cite{mw74} showed that
$I_\alpha$ is bounded from $L^p(w^p)$ to $L^q(w^q)$
if and only if $w$ satisfies the $A_{p,q}$ condition:
\begin{align}\label{eq-Apq}
[w]_{A_{p,q}}:=\sup_{Q\subset\mathbb{R}^n}
\fint_Q [w(x)]^q\,dx
\left(\fint_Q [w(x)]^{-p'}\,dx\right)^{\frac{q}{p'}}<\infty,
\end{align}
where the supremum is taken over all cubes $Q\subset\mathbb R^n$
and, for any $r\in(1,\infty)$, $r':=\frac{r}{r-1}$ denotes the
\emph{conjugate index} of $r$.

Once these weighted boundedness characterizations were established, a
natural problem is to determine the optimal
dependence of the above boundedness on the corresponding weight constants.
Initially, for any $p\in(1,\infty)$ and $w\in A_p$,
Buckley \cite{b93} proved that
\begin{align*}
\|M\|_{L^p(w)\to L^p(w)}
\lesssim [w]_{A_p}^{\frac{1}{p-1}},
\end{align*}
where the implicit positive constant is independent of
$w$ and the exponent $\frac{1}{p-1}$ is \emph{optimal}.
Such quantitative estimates were also extensively studied for
Calder\'on--Zygmund operators. For any
$p\in(1,\infty)$ and $w\in A_p$, Petermichl obtained the
quantitative boundedness of the Hilbert transform on $L^p(w)$
in \cite{p07} and of the Riesz transforms on $L^p(w)$ in \cite{p08},
with the optimal dependence on the weight constant
$[w]_{A_p}^{\max\{1,\frac{1}{p-1}\}}$.
For general Calder\'on--Zygmund operators, this dependence on
the weight constants was completely settled by Hyt\"onen \cite{h12}.
In the off-diagonal setting, if $\alpha$, $p$, and $q$
satisfy \eqref{eq:range} and if $w\in A_{p,q}$, then
Lacey et al. \cite[Theorem 2.6]{lmpt10}
proved that, for any $f\in L^p(w^p)$,
\begin{align}\label{eq:scalar}
\|I_\alpha f\|_{L^q(w^q)}\lesssim
[w]_{A_{p,q}}^{(1-\frac{\alpha}{n})\max\{1,\frac{p'}{q}\}}
\|f\|_{L^p(w^p)},
\end{align}
where the implicit positive constant is independent of
$w$ and the exponent $(1-\frac{\alpha}{n})\max\{1,\frac{p'}{q}\}$ is optimal.
For more details and further developments about
sharp weighted estimates for more operators,
we refer to \cite{cu25,cmp12,cm13,pereyra18}.

The study of matrix weights can be traced back to the work of
Wiener and Masani \cite[Section 4]{wm58}. Let $d\in\mathbb{N}$.
A matrix-valued function $W$ on $\mathbb{R}^n$, taking values in
the set of $d\times d$ complex positive semidefinite
matrices, is called a \emph{matrix weight}
if $W$ is positive definite almost everywhere and all the
entries of $W$ are locally integrable functions on $\mathbb{R}^n$.
Let $W$ be a matrix weight on $\mathbb{R}^n$. For any $p\in[1,\infty)$,
the \emph{matrix-weighted Lebesgue space} $L^p(W)$ is
defined to be the set of all measurable vector-valued
functions $\vec{f}: \mathbb{R}^n\to \mathbb{C}^d$ such that
\begin{align}\label{eq-LpW}
\left\|\vec{f}\right\|_{L^{p}(W)}:=\left[\int_{\mathbb{R}^n}
\left|W^{\frac{1}{p}}(x)\vec{f}(x)\right|^{p}\,dx
\right]^{\frac{1}{p}}<\infty.
\end{align}
In the 1990s, to study multivariate random stationary
processes and the invertibility of Toeplitz operators,
Treil and Volberg \cite{tv97} determined the
matrix $\mathscr{A}_2$ condition on $\mathbb{R}$ and proved that
the Hilbert transform is bounded on $L^2(W)$ if and only if
$W\in\mathscr{A}_2$. For $p\in(1, \infty)$,
Nazarov and Treil \cite{nt96}
and Volberg \cite{v97} independently formulated the matrix
$\mathscr{A}_p$ condition on $\mathbb{R}$ and showed that the Hilbert transform is bounded on
$L^p(W)$ if and only if $W\in\mathscr{A}_p$.
Moreover, for any $p\in(1,\infty)$ and $W\in \mathscr{A}_p$,
Christ and Goldberg \cite{cg01,g03} further established
the boundedness of the Christ--Goldberg maximal operator
(an analogue of the Hardy--Littlewood maximal operator in
the matrix-weighted setting) and
Calder\'on--Zygmund operators on $L^{p}(W)$.

As in the scalar setting, a fundamental problem is to
determine the optimal dependence of weighted norm inequalities
on the corresponding matrix weight constants.
However, the optimal dependence of the
matrix-weighted boundedness of the Hilbert transform
on the weight constants differs from its scalar counterpart.
Using the technique of convex body domination,
Nazarov et al. \cite{nptv17} proved that the
Calder\'on--Zygmund operators are bounded on $L^2(W)$
with dependence on the weight constant
$[W]^{\frac{3}{2}}_{\mathscr{A}_2}$.
For general $p\in(1,\infty)$ and $W\in \mathscr{A}_p$,
Cruz-Uribe et al. \cite[Corollary 1.16]{cim18}
established the boundedness of Calder\'on--Zygmund operators
on $L^p(W)$ with the dependence on the weight constant
$[W]^{1+\frac{1}{p(p-1)}}_{\mathscr{A}_p}$.
Surprisingly, when $p=2$, the exponent $\frac{3}{2}$
of $[W]_{\mathscr{A}_2}$ was proved to
be optimal by Domelevo et al. in \cite{dptv24} for the Hilbert transform.
Furthermore, the optimality of the exponent $1+\frac{1}{p(p-1)}$
of $[W]_{\mathscr{A}_p}$ for general
$p\in(1,\infty)$ was also established by Jiao et al. in \cite{jqwx26} for the
Hilbert transform. More recently, Bu et al. \cite{byy} obtained
the optimality of the exponent $1+\frac{1}{p(p-1)}$ of
$[W]_{\mathscr{A}_p}$ appearing in the boundedness of
non-degenerate convolution Calder\'on--Zygmund operators on $L^p(W)$
with $p\in(1,\infty)$.
Thus, for every $p\in(1,\infty)$, the optimal exponent of
$[W]_{\mathscr A_p}$ in the $L^p(W)$ bound for the Hilbert transform
is strictly larger than the corresponding exponent of $[w]_{A_p}$
in the scalar setting. In view of this, it is \emph{natural} to ask whether the
optimal exponent of $[W]_{\mathscr A_{p,q}}$ in the matrix-weighted
boundedness of fractional integrals also differs from the scalar one
in \eqref{eq:scalar}. Surprisingly, the first main result of this article
is to give a negative answer to this question, which is different from
the aforementioned Calder\'on--Zygmund operator case.

In what follows, for any $k,m\in\mathbb N$ and any $k\times m$ complex matrix $A$,
we use $\|A\|$ to denote its operator norm from $\mathbb C^m$ to $\mathbb C^k$.
Let $1<p\leq q<\infty$. The \emph{matrix Muckenhoupt class
$\mathscr A_{p,q}$} is defined to be the set of
all matrix weights $W$ such that
\begin{align}\label{eq:Apq}
[W]_{\mathscr A_{p,q}}:=\sup_{Q\subset\mathbb{R}^n}\fint_Q\left[
\fint_Q\left\|W(x)W^{-1}(y)\right\|^{p'}\,dy
\right]^{\frac{q}{p'}}\,dx<\infty,
\end{align}
where the supremum is taken over all cubes $Q\subset\mathbb R^n$.
Let $\alpha$, $p$, and $q$ satisfy \eqref{eq:range}, and let $W\in \mathscr A_{p,q}$.
Isralowitz and Moen in \cite[Theorem 1.4]{im19} proved that,
for any $\vec{f}\in L^p(W^p)$,
\begin{align}\label{eq-integral-im}
\left\|I_\alpha \vec{f}\right\|_{L^q(W^q)}
\lesssim[W]_{\mathscr A_{p,q}}^{(1-\frac{\alpha}{n})
\frac{p'}{q}+\frac{1}{q'}}\left\|\vec{f}\right\|_{L^p(W^p)},
\end{align}
where the implicit positive constant is independent of $\vec{f}$ and $W$.
Note that $1-\frac{\alpha}{n}=\frac{1}{p'}+\frac{1}{q}\in(0,1)$,
and hence $(1-\frac{\alpha}{n})\max\{1,\frac{p'}{q}\}=
\max\{\frac{1}{p'}+\frac{1}{q},\frac{1}{q}+\frac{p'}{q^2}\}
<1+\frac{p'}{q^2}=(1-\frac{\alpha}{n})
\frac{p'}{q}+\frac{1}{q'}$, which further implies that
the exponent of $[W]_{\mathscr A_{p,q}}$
in \eqref{eq-integral-im} is strictly larger than the exponent
of $[w]_{A_{p,q}}$ in \eqref{eq:scalar}
for all $\alpha$, $p$, and $q$ satisfying \eqref{eq:range}.
Furthermore, two matrix-weighted estimates for fractional
integrals and their commutators have also been studied
in \cite{ci22,cim18}.
More recently, by establishing the fractional
convex body domination of $I_\alpha$,
Cruz-Uribe et al. \cite[Theorem 1.2]{clm26} further improved
the exponent of $[W]_{\mathscr A_{p,q}}$ in \eqref{eq-integral-im} into
\begin{align}\label{eq-exponent}
\min\left\{
\left(1-\frac{\alpha}{n}\right)\frac{p'}{q}+\frac{1}{q'},\,
1-\frac{\alpha}{n}+\frac{p'-1}{q}
\right\}.
\end{align}
A simple computation still shows that
$1-\frac{\alpha}{n}+\frac{p'-1}{q}$
is strictly larger than the sharp scalar exponent,
and hence the exponent in \eqref{eq-exponent} is still strictly
larger than the sharp scalar exponent in \eqref{eq:scalar} for all
$\alpha$, $p$, and $q$ satisfying \eqref{eq:range}.
Therefore, it is a \emph{natural question} to ask whether the exponent in \eqref{eq-exponent}
can be further improved or whether it is already optimal.
In the following theorem, we prove a quantitative matrix-weighted
boundedness of fractional integrals, in which the exponent
of the weight constant coincides with the corresponding
scalar exponent in \eqref{eq:scalar}, and hence is \emph{optimal}.
For a vector-valued function $\vec f$,
$I_\alpha \vec f$ is defined by applying $I_\alpha$
to each component of $\vec f$.

\begin{theorem}\label{thm:main}
Let $\alpha$, $p$, and $q$ satisfy \eqref{eq:range},
and let $W\in\mathscr{A}_{p,q}$.
Then, for any $\vec{f}\in L^p(W^p)$,
\begin{align}\label{eq:main}
\left\|I_\alpha \vec{f}\right\|_{L^q(W^q)}
\lesssim[W]_{\mathscr A_{p,q}}^{(1-\frac{\alpha}{n})\max\{1,\frac{p'}{q}\}}
\left\|\vec{f}\right\|_{L^p(W^p)},
\end{align}
where the implicit positive constant is independent of
$W$ and $\vec{f}$. Moreover, the exponent
$(1-\frac{\alpha}{n})\max\{1,\frac{p'}{q}\}$ of the weight constant in
\eqref{eq:main} is optimal.
\end{theorem}

To prove Theorem \ref{thm:main}, we introduce an \emph{off-diagonal chain-packing principle of dyadic cubes}, which is presented in Theorem \ref{thm-chain-packing}
whose proof is inspired by that of Lerner \cite[Theorem 1.3]{l26}. This chain-packing
principle strongly depends on the off-diagonal assumption
$p<q$. Based on this argument, we reduce the proof of
\eqref{eq:main} to the estimate of sums of some double integrals over
chains. To deal with this estimate over chains, we borrow some ideas from
the proof of \cite[Lemma 5.7]{nptv17} and further establish
the off-diagonal Carleson-type embedding
in Lemma \ref{lem:off-diagonal-Carleson}. Let $I_d$ denote the identity
matrix of order $d$. In particular, if $W=wI_d$, then
$[W]_{\mathscr A_{p,q}}=[w]_{A_{p,q}}$.
Using this observation, we deduce the optimality of the exponent in \eqref{eq:main}
from the scalar one in \eqref{eq:scalar}.

In the scalar setting, the boundedness of fractional integrals
is closely related to Sobolev inequalities. Indeed,
it is well known that, for any sufficiently smooth $f$ on
$\mathbb{R}^n$ with compact support and for any $x\in\mathbb{R}^n$,
\begin{align}\label{eq-f-I1-grdientf}
|f(x)|\lesssim I_1(|\nabla f|)(x).
\end{align}
It immediately follows from this inequality and
\eqref{eq:scalar} with $\alpha=1$ that
a quantitative weighted Sobolev inequality holds.
However, as observed in \cite[Remark 2.8]{lmpt10},
the dependence of weight constants in the
quantitative weighted Sobolev inequality
obtained in this way can be further improved.
To be precise, to prove \eqref{eq:scalar},
Lacey et al. \cite[Theorem 2.4]{lmpt10} established the sharp weak-type
estimate for fractional integrals. Using this weak-type
estimate together with a truncation argument based on scalar level sets,
Lacey et al. further obtained a quantitative weighted Sobolev inequality.
To state this result, we need to recall the definition of
$A_{1,q}$ weights in \cite[(3.2)]{lmpt10} for $q\in(1,\infty)$.
For any $q\in(1,\infty)$, the class $A_{1,q}$ is defined to
be the set of all weights $w$ such that
\begin{align}\label{eq:A1q-scalar}
[w]_{A_{1,q}}:=\sup_{Q\subset\mathbb R^n}
\fint_Q [w(x)]^q\,dx\,
\mathop{\mathrm{ess\,sup}}_{y\in Q}[w(y)]^{-q}<\infty,
\end{align}
where the supremum is taken over all cubes $Q\subset\mathbb R^n$.
Let $n\geq2$ be an integer, $p\in[1,n)$,
$q:=\frac{np}{n-p}$, and $w\in A_{p,q}$.
In what follows, we use the \emph{symbol $f$} to denote a complex-valued
function on $\mathbb R^n$
and the \emph{symbol $\vec f$} to denote a $\mathbb C^d$-valued function on $\mathbb R^n$.
Thus, for any $p\in[1,\infty)$, we write both $L^p(w)$
when $w\equiv1$ and $L^p(W)$ when $W\equiv I_d$, respectively, defined in \eqref{eq-Lp}
and \eqref{eq-LpW}, simply as $L^p$.
It was proved in \cite[Theorem 2.7]{lmpt10} that,
for any compactly supported Lipschitz function $f$,
\begin{align}\label{eq:scalar-sobolev-intro}
\|wf\|_{L^{q}}\lesssim
[w]_{A_{p,q}}^{\frac{n-1}{n}}
\|w\nabla f\|_{L^p},
\end{align}
where the implicit positive constant is independent of
$w$ and $f$.
Furthermore, the exponent $\frac{n-1}{n}$ in
\eqref{eq:scalar-sobolev-intro} was later proved to be optimal by
Cruz-Uribe and Moen \cite[Theorem 1.7]{cm12}.

In what follows, we always use $d\in\mathbb{N}$ to
denote the dimension of vectors and the order of square matrices.
Suppose that $\vec f=(f_1,\dots,f_d)$ is a continuously differentiable function
from $\mathbb R^n$ to $\mathbb C^d$. For any $x\in\mathbb{R}^n$, let
\begin{align*}
D\vec f(x):=
\begin{pmatrix}
\partial_1 f_1(x) & \cdots & \partial_n f_1(x)\\
\vdots & \ddots & \vdots\\
\partial_1 f_d(x) & \cdots & \partial_n f_d(x)
\end{pmatrix}
\end{align*}
be the \emph{Jacobian matrix} of $\vec f$ at $x$, where
$\partial_i:=\frac{\partial}{\partial x_i}$ for $i\in\{1,\ldots,n\}$.
For any $d\times n$ complex matrix-valued function $F$ on
$\mathbb{R}^n$ and for any $p\in[1,\infty)$, let
\begin{align*}
\|F\|_{L^p(\mathbb R^n,\mathbb C^{d\times n})}
:=\left[\int_{\mathbb{R}^n}\|F(x)\|^p\,dx\right]^{\frac{1}{p}}.
\end{align*}
Let $n\geq2$ be an integer,
$p\in(1,n)$, $q:=\frac{np}{n-p}$, and $W\in\mathscr{A}_{p,q}$.
Based on the argument used in the proof of \eqref{eq-integral-im},
Isralowitz and Moen \cite[Theorem 4.2]{im19} proved that,
for any Schwartz function $\vec{f}$,
\begin{align*}
\left\|\vec{f}\right\|_{L^{q}(W^q)}
\lesssim[W]_{\mathscr A_{p,q}}^{
\frac{(n-1)p'}{nq}+\frac{1}{q'}}
\left\|WD\vec{f}\right\|_{L^p(\mathbb R^n,\mathbb C^{d\times n})},
\end{align*}
where the implicit positive constant is independent of
$W$ and $\vec{f}$. Let $C_{\rm c}^\infty$ be the set of all
infinitely differentiable functions
$\vec{f}:\mathbb{R}^n\to\mathbb C^d$ with compact support.
More recently, using convex body domination and matrix-weighted
estimates for fractional singular integral operators,
Cruz-Uribe et al. \cite[Theorem 1.5]{clm26} showed that,
for any $\vec{f}\in C_{\rm c}^\infty$,
\begin{align}\label{eq-sobolev-clm}
\left\|\vec{f}\right\|_{L^{q}(W^q)}
\lesssim[W]_{\mathscr A_{p,q}}^{
\min\{\frac{(n-1)p'}{nq}+\frac{1}{q'},\,
\frac{n-1}{n}+\frac{p'-1}{q}\}}
\left\|WD\vec{f}\right\|_{L^p(\mathbb R^n,\mathbb C^{d\times n})}.
\end{align}
Note that the exponent of $[W]_{\mathscr A_{p,q}}$
in \eqref{eq-sobolev-clm} is also strictly
larger than $\frac{n-1}{n}$ in \eqref{eq:scalar-sobolev-intro}.
Our second result further improves the exponent
of $[W]_{\mathscr A_{p,q}}$ in \eqref{eq-sobolev-clm}
to $\frac{n-1}{n}$, which is the same as the one
in \eqref{eq:scalar-sobolev-intro} and hence optimal.
To present this result, we need the following concept.
For any $q\in(1,\infty)$, the \emph{class $\mathscr A_{1,q}$} is defined to be
the set of all  matrix weights $W$ such that
\begin{align}\label{eq:A1q}
[W]_{\mathscr A_{1,q}}
:=\sup_{Q\subset\mathbb{R}^n}\fint_Q
\mathop{\mathrm{ess\,sup}}_{y\in Q}
\left\|W(x)W^{-1}(y)\right\|^q\,dx<\infty,
\end{align}
where the supremum is taken over all cubes $Q\subset\mathbb R^n$.
Our second main result is the following optimal
matrix-weighted Sobolev inequality.

\begin{theorem}\label{thm:matrix-weighted-Sobolev}
Let $n\geq2$, $p\in[1,n)$, and $q:=\frac{np}{n-p}$.
If $W\in\mathscr A_{p,q}$, then, for any
$\vec{f}\in C_{\rm c}^\infty$,
\begin{align}\label{eq:sobolev}
\left\|\vec{f}\right\|_{L^{q}(W^q)}
\lesssim[W]_{\mathscr A_{p,q}}^{\frac{n-1}{n}}
\left\|WD\vec{f}\right\|_{L^p(\mathbb R^n,\mathbb C^{d\times n})},
\end{align}
where the implicit positive constant is independent of
$W$ and $\vec{f}$. Moreover, the exponent
$\frac{n-1}{n}$ of the weight constant in
\eqref{eq:sobolev} is optimal.
\end{theorem}

As in the scalar setting, the sharp matrix-weighted boundedness of fractional integrals in \eqref{eq:main} and its proof cannot be
directly used to prove \eqref{eq:sobolev} for the full range
$p\in[1,n)$ (see Remark \ref{rmk-sobolev} for more details).
Moreover, neither the pointwise inequality
\eqref{eq-f-I1-grdientf} nor the scalar truncation argument
in the proof of \cite[Theorem 2.7]{lmpt10}
has a suitable analogue in the matrix-weighted setting.
Therefore, to show Theorem \ref{thm:matrix-weighted-Sobolev},
we use a stopping-time construction
adapted to convex sets and reducing operators.
This construction yields a sparse domination
of $|W\vec f|$, whose coefficients involve
the local $L^p$-norms of $WD\vec f$ over pairwise disjoint sets.
The desired matrix-weighted Sobolev inequality then follows
from this sparse domination and the off-diagonal Carleson-type embedding
in Lemma \ref{lem:off-diagonal-Carleson}.

The remainder of this article is organized as follows. In Section
\ref{sec:fractional}, via reducing operators, we first present
the quantitative reverse H\"{o}lder inequality
for the class $\mathscr{A}_{p,q}$ in Lemma \ref{lem:op-RH}.
Using the shifted dyadic grids and duality,
we reduce the proof of Theorem \ref{thm:main} to an estimate
of sums of some double integrals over shifted dyadic grids,
which is the estimate stated in Theorem \ref{thm:dyadic-form}.
To prove Theorem \ref{thm:dyadic-form}, we give an off-diagonal chain-packing
principle of dyadic cubes in Theorem \ref{thm-chain-packing},
which further reduces the aforementioned estimate to the corresponding one
over chains. Using the off-diagonal Carleson-type
embedding result in Lemma \ref{lem:off-diagonal-Carleson},
we show the chain estimate in Lemma \ref{lem-simple-chain}.
Finally, we establish Theorem \ref{thm:dyadic-form} and hence
Theorem \ref{thm:main}. In Section \ref{sec:sobolev}, given a
continuously differentiable vector-valued function
and a nonempty closed convex set, we first construct
in Lemma \ref{lem:Sobolev-smooth-truncation}
a continuously differentiable scalar
function whose values depend on the distance between
the values of the given vector-valued function and the convex set.
Next, via a stopping-time construction
associated with convex sets and reducing operators,
we prove Theorem \ref{thm:matrix-weighted-Sobolev}.

We end this introduction with some notational conventions.
Throughout this article, we work in $\mathbb{R}^n$ and, unless
otherwise specified, take $\mathbb{R}^n$ as the underlying space.
For any $z\in\mathbb C$, let $\overline z$ denote the
\emph{complex conjugate} of $z$.
For any $\vec x:=(x_1,\dots,x_d),
\vec y:=(y_1,\dots,y_d)\in\mathbb C^d$, let
\begin{align*}
\left(\vec x,\vec y\right):=\sum_{j=1}^d x_j\overline{y_j}
\text{ and }|\vec x|:=\left(\vec x,\vec x\right)^{\frac12}.
\end{align*}
Let $\mathbb N:=\{1,2,\dots\}$ and
$\mathbb Z_+:=\mathbb N\cup\{0\}$.
For any cube $Q\subset\mathbb{R}^n$, let $\ell(Q)$ denote
the edge length of $Q$. For any measurable set $E\subset\mathbb{R}^n$,
let $|E|$ denote its Lebesgue measure
and $\mathbf1_E$ its characteristic function. If
$|E|\in(0,\infty)$, for any measurable
function $f$ on $\mathbb{R}^n$, let
\begin{align*}
\fint_E f(x)\,dx:=\frac{1}{|E|}\int_E f(x)\,dx
\end{align*}
whenever the integral is well defined.
Let $\mathbf{0}$ denote the \emph{origin} of $\mathbb{C}^d$.
When $p=1$, we use the convention $p'=\infty$ and hence $\frac{1}{p'}=0$.
The \emph{symbol $C$} denotes a positive constant which is independent
of the main parameters involved, but may vary from line to line.
The \emph{notation $A\lesssim B$} means that $A\leq CB$ for some positive constant $C$,
while $A\sim B$ means $A\lesssim B\lesssim A$.
Finally, in all subsequent proofs we retain the notation introduced
in the relevant statement.

\section{Proof of Theorem \ref{thm:main}}\label{sec:fractional}

This section is devoted to proving Theorem \ref{thm:main}.
Specifically, in Subsection \ref{sec:reducing},
we recall the concept of reducing operators and
the quantitative reverse H\"older inequality for classes $\mathscr{A}_{p,q}$.
In Subsection \ref{sec:chain-packing}, using shifted dyadic grids and
duality, we reduce the proof of Theorem \ref{thm:main} to an estimate of
sums of some double integrals over shifted dyadic grids.
Next, we prove the chain-packing principle,
which further reduces the aforementioned estimate to the corresponding one
over chains. In Subsection \ref{sec:chain-estimates}, we establish
this chain estimate and finally show Theorem \ref{thm:main}.

\subsection{Reducing Operators and Quantitative
Reverse H\"{o}lder Estimates}\label{sec:reducing}

In this subsection, we first recall the concept of reducing operators
and the relationship between the weight constant $[W]_{\mathscr{A}_{p,q}}$
and the reducing operators of $W$. Finally, via reducing operators,
we present the quantitative reverse H\"older inequalities for classes $\mathscr{A}_{p,q}$.

We begin by recalling the concept of reducing operators,
which was originally introduced by Volberg \cite[(3.1)]{v97}.
Let $q\in(1,\infty)$, and let $W$ be a matrix weight such that
$\|W\|^q$ is locally integrable.
From \cite[Proposition 1.2]{g03},
we deduce that, for any bounded measurable set $E\subset\mathbb{R}^n$
with positive measure, there exists a positive definite matrix
$A_E$ such that, for any $\vec z\in\mathbb{C}^d$,
\begin{align*}
\left[\fint_E\left|W(x)\vec{z}
\right|^q\,dx\right]^{\frac{1}{q}}\leq
\left|A_E\vec z\right|\leq \sqrt{d}\left[\fint_E
\left|W(x)\vec z\right|^q\,dx\right]^{\frac{1}{q}}.
\end{align*}
We call $A_E$ a \emph{reducing operator}
of order $q$ for $W$.
By the norm equivalence theorem in finite-dimensional spaces,
we find that, if $A_E$ is a reducing operator
of order $q$ for $W$, then, for any $d\times d$
complex matrix $M$,
\begin{align}\label{eq:reduce}
\|A_E M\|\sim\left[\fint_E\left\|W(x)M
\right\|^q\,dx\right]^{\frac{1}{q}},
\end{align}
where the positive equivalence constants depend only on $d$ and $q$.
Note that the discussions in \cite[Section 2]{im19} and
\cite[Corollary 3.4]{im19} yield the following
duality property of $\mathscr A_{p,q}$ and a
description of its weight constant in terms of reducing operators.
\begin{proposition}\label{prop-weight-reducing}
Let $1<p\leq q<\infty$, and let $W\in\mathscr A_{p,q}$.
Then $W^{-1}\in\mathscr{A}_{q',p'}$.
Moreover, for each cube $Q\subset\mathbb R^n$, let $A_Q$ and
$\widetilde{A}_Q$ be reducing operators, respectively,
of order $q$ for $W$ and of order $p^{\prime}$ for $W^{-1}$.
Then
\begin{align*}
[W]_{\mathscr{A}_{p,q}}^{\frac{1}{q}}
&\sim\sup_Q\left\|A_Q\widetilde{A}_Q\right\|,
\end{align*}
where the positive equivalence constants depend only on $p$, $q$, and $d$.
\end{proposition}

It is well known that any complex positive semidefinite
matrix is always Hermitian (see, for instance, \cite[Theorem 4.1.4]{hj13}).
The following basic identity for matrix norms of
products of Hermitian matrices follows from
\cite[Theorem 5.6.2(d)]{hj13}; we omit the details.
\begin{lemma}\label{lem-norm-AB=BA}
Suppose that $U,V$ are two $d\times d$
complex Hermitian matrices. Then $\|UV\|=\|VU\|$.
\end{lemma}
\begin{remark}\label{rmk-dual}
The following three statements hold.
\begin{itemize}
\item[{\rm (i)}] Under the assumptions of Proposition \ref{prop-weight-reducing},
it follows from Lemma \ref{lem-norm-AB=BA} that
\begin{align*}
\left[W^{-1}\right]_{\mathscr{A}_{q',p'}}^{\frac{1}{p'}}
\sim[W]_{\mathscr{A}_{p,q}}^{\frac{1}{q}},
\end{align*}
where the positive equivalence constants depend only on $p$, $q$, and $d$.
\item[{\rm (ii)}] Let $1<p\leq q<\infty$, and let $W\in\mathscr A_{p,q}$.
For any cube $Q$ in $\mathbb R^n$, let $A_Q$ and $\widetilde A_Q$
be reducing operators, respectively, of order $q$ for $W$
and of order $p'$ for $W^{-1}$.
By \eqref{eq:reduce} and Lemma \ref{lem-norm-AB=BA}, we find that
\begin{align}\label{eq:dual-reduce}
\left[\fint_Q\left\|A_QW^{-1}(y)\right\|^{p'}\,dy\right]^{\frac1{p'}}
\sim\left\|\widetilde A_QA_Q\right\|
=\left\|A_Q\widetilde A_Q\right\|
\lesssim[W]_{\mathscr{A}_{p,q}}^{\frac{1}{q}},
\end{align}
where the implicit positive constants depend only on $p$, $q$, and $d$.

\item[{\rm (iii)}]
Let $q\in(1,\infty)$ and $W\in \mathscr{A}_{1,q}$.
For any cube $Q$ in $\mathbb R^n$, let $A_Q$
be a reducing operator of order $q$ for $W$.
From \eqref{eq:reduce} and the definition of
$[W]_{\mathscr A_{1,q}}$, we deduce that, for almost
every $y\in Q$,
\begin{align*}
\left\|A_QW^{-1}(y)\right\|\sim
\left[\fint_Q\left\|W(x)W^{-1}(y)
\right\|^q\,dx\right]^{\frac1q}
\leq[W]_{\mathscr A_{1,q}}^{\frac{1}{q}},
\end{align*}
and hence
\begin{align}\label{eq:end-reduce}
\mathop{\operatorname{ess\,sup}}_{y\in Q}
\left\|A_QW^{-1}(y)\right\|
\lesssim[W]_{\mathscr A_{1,q}}^{\frac{1}{q}}.
\end{align}
\end{itemize}
\end{remark}

We next give the quantitative reverse H\"older property of
the weight class $\mathscr{A}_{p,q}$. For this purpose, let
$\alpha$, $p$, and $q$ satisfy \eqref{eq:range},
and let $W\in \mathscr{A}_{p,q}$.
From the proof of \cite[Lemma 2.8]{hyyz26},
we infer that, for any $\vec{z}\in\mathbb{C}^d\setminus\{\mathbf{0}\}$,
$w_{\vec{z}}:=|W\vec{z}|\in A_{p,q}$ and
$[w_{\vec{z}}]_{A_{p,q}}\leq[W]_{\mathscr{A}_{p,q}}$.
By H\"older's inequality and the assumption $p<q$,
we find that $w_{\vec{z}}^q\in A_q$
and $[w_{\vec{z}}^q]_{A_q}\leq[w_{\vec{z}}]_{A_{p,q}}$.
It follows from \cite[Proposition 7.1.5(5)]{g14c} that,
for any $w\in A_q$, $[w]_{A_q}\geq1$, and hence,
for any $\vec{z}\in\mathbb{C}^d\setminus\{\mathbf{0}\}$,
\begin{align}\label{eq:Apq-lower}
1\leq[w_{\vec{z}}]_{A_{p,q}}\leq[W]_{\mathscr{A}_{p,q}}.
\end{align}
Using this observation and applying the quantitative reverse
H\"{o}lder inequality in \cite[Theorem 2.3]{hp13} to
$w_{\vec{z}}^q$, we obtain the following
quantitative reverse H\"{o}lder inequality for $W$.
\begin{lemma}\label{lem-reverseHolder}
Let $\alpha$, $p$, and $q$ satisfy \eqref{eq:range},
and let $W\in \mathscr{A}_{p,q}$.
Then there exists a positive constant
$\tau_n\in[1,\infty)$ depending only on
$n$ such that, for any $r\in(1,1+\frac{1}{\tau_n[W]_{\mathscr{A}_{p,q}}}]$,
any $\vec{z}\in\mathbb{C}^d\setminus\{\mathbf{0}\}$,
and any cube $Q$ in $\mathbb{R}^n$,
\begin{align*}
\left[\fint_Q|W(x)\vec{z}|^{qr}\,dx\right]^{\frac{1}{r}}
\leq 2\fint_Q|W(x)\vec{z}|^{q}\,dx.
\end{align*}
\end{lemma}

Let $\alpha$, $p$, and $q$ satisfy \eqref{eq:range},
and let $W\in\mathscr A_{p,q}$. Let
\begin{align}\label{eq:epsilon}
\varepsilon_W
:=
\frac{1}{\tau_n[W]_{\mathscr A_{p,q}}},
\end{align}
where $\tau_n$ is the constant in Lemma \ref{lem-reverseHolder}.
The following lemma can be regarded as a matrix version of
Lemma \ref{lem-reverseHolder}.
\begin{lemma}\label{lem:op-RH}
Let $\alpha$, $p$, and $q$ satisfy \eqref{eq:range},
let $W\in\mathscr A_{p,q}$, and, for any cube
$Q\subset\mathbb R^n$, let $A_Q$ be a reducing operator
of order $q$ for $W$. Then, for any cube
$Q\subset\mathbb R^n$,
\begin{align}\label{eq:op-RH}
\fint_Q\left\|W(x)A_Q^{-1}\right\|^{q(1+\varepsilon_W)}\,dx
\lesssim\left[\fint_Q
\left\|W(x)A_Q^{-1}\right\|^{q}\,dx\right]^{1+\varepsilon_W},
\end{align}
where the implicit positive constant depends only on
$n$, $d$, and $q$.
\end{lemma}

\begin{proof}
Let $Q$ be a cube in $\mathbb R^n$, and let
$\{e_j\}_{j=1}^d$ be an orthonormal basis of $\mathbb C^d$.
By \eqref{eq:Apq-lower}, we find that $[W]_{\mathscr A_{p,q}}\geq1$
and hence $\varepsilon_W\in(0,\frac{1}{\tau_n}]$,
where $\tau_n$ is the constant in Lemma \ref{lem-reverseHolder}.
Applying the norm equivalence theorem in finite-dimensional spaces,
we conclude that, for any $d\times d$ complex matrix $M$,
\begin{align*}
\|M\|^{q(1+\varepsilon_W)}
\lesssim\sum_{j=1}^d
|Me_j|^{q(1+\varepsilon_W)},
\end{align*}
where the implicit positive constant depends only on $n$, $d$, and $q$.
Hence,
\begin{align}\label{eq:vector-RH}
\fint_Q\left\|W(x)A_Q^{-1}\right\|^{q(1+\varepsilon_W)}\,dx
\lesssim\sum_{j=1}^d\fint_Q
\left|W(x)A_Q^{-1}e_j\right|^{q(1+\varepsilon_W)}\,dx.
\end{align}
For any $j\in\{1,\dots,d\}$, using Lemma \ref{lem-reverseHolder} with
$\vec z:=A_Q^{-1}e_j$ and $r:=1+\varepsilon_W$, we obtain
\begin{align*}
\fint_Q\left|W(x)A_Q^{-1}e_j\right|^{q(1+\varepsilon_W)}\,dx
\leq2^{1+\varepsilon_W}\left[\fint_Q
\left|W(x)A_Q^{-1}e_j\right|^q\,dx\right]^{1+\varepsilon_W}.
\end{align*}
This, together with \eqref{eq:vector-RH}, further implies that
\begin{align*}
\fint_Q\left\|W(x)A_Q^{-1}\right\|^{q(1+\varepsilon_W)}\,dx
&\lesssim\sum_{j=1}^d\left[\fint_Q
\left|W(x)A_Q^{-1}e_j\right|^q\,dx\right]^{1+\varepsilon_W}\\
&\leq\left[\sum_{j=1}^d\fint_Q\left|W(x)A_Q^{-1}e_j\right|^q\,dx
\right]^{1+\varepsilon_W}\\
&\lesssim\left[\fint_Q\left\|W(x)A_Q^{-1}
\right\|^q\,dx\right]^{1+\varepsilon_W},
\end{align*}
where the last inequality follows from
the norm equivalence theorem in
finite-dimensional spaces again.
This completes the proof of \eqref{eq:op-RH}
and hence Lemma \ref{lem:op-RH}.
\end{proof}

\subsection{Chain-Packing Method}\label{sec:chain-packing}

In this subsection, using the shifted dyadic grids and duality,
we first reduce the proof of \eqref{eq:main} to the estimate
of sums of some double integrals over shifted dyadic grids
in Theorem \ref{thm:dyadic-form}.
Inspired by the proof of \cite[Theorem 1.3]{l26},
we give an off-diagonal chain-packing principle of dyadic cubes
in Theorem \ref{thm-chain-packing}, which further reduces the above estimate to
the corresponding estimate over chains.

Recall that, for any $\tau\in\{0,\frac13,\frac23\}^n$, the
\emph{shifted dyadic grid $\mathcal{D}^\tau$} is defined by setting
\begin{align}\label{eq:shifted-grid}
\mathcal{D}^\tau
:=\left\{2^j\left[m+[0,1)^n+(-1)^j\tau\right]:\,
j\in\mathbb Z,\ m\in\mathbb Z^n\right\}.
\end{align}
Let $\mathfrak D:=\bigcup_{\tau\in\{0,\frac13,\frac23\}^n}\mathcal{D}^\tau$.
The following properties of shifted dyadic grids are also needed
(see, for instance, \cite[p.\,479]{mtt02}).

\begin{lemma}\label{lem:shifted-dyadic}
The following two assertions hold.
\begin{itemize}
\item[\rm (i)] For any $P,Q\in\mathcal{D}^\tau$, with
$\tau\in\{0,\frac13,\frac23\}^n$, one has
$P\cap Q\in\{\emptyset,P,Q\}$.
\item[\rm (ii)] For any cube $Q$ in $\mathbb R^n$, there exists
$P\in\mathfrak D$ such that $Q\subset P$ and
$\ell(P)\in(\frac32\ell(Q),3\ell(Q)]$.
\end{itemize}
\end{lemma}

Using the shifted dyadic grids, we obtain
the following pointwise estimate for the fractional kernel.

\begin{lemma}\label{lem:kernel-dyadic}
Let $\alpha\in(0,n)$. For any $x,y\in\mathbb R^n$ with $x\neq y$,
\begin{align*}
\frac{1}{|x-y|^{n-\alpha}}
\lesssim
\sum_{\tau\in\{0,\frac13,\frac23\}^n}
\sum_{Q\in\mathcal{D}^\tau}
|Q|^{\frac{\alpha}{n}-1}\mathbf{1}_Q(x)\mathbf{1}_Q(y),
\end{align*}
where the implicit positive constant depends only on $n$ and $\alpha$.
\end{lemma}

\begin{proof}
Fix  $x,y\in\mathbb R^n$ with $x\neq y$.
Choose a cube $R\subset \mathbb{R}^n$ containing
$x$ and $y$ such that $\ell(R)=2|x-y|$. By Lemma
\ref{lem:shifted-dyadic}(ii), there exist
$\tau\in\{0,\frac13,\frac23\}^n$ and $Q\in\mathcal{D}^\tau$ such that
$R\subset Q$ and $\ell(Q)\leq3\ell(R)=6|x-y|$.
From this and the assumption that $\alpha\in(0,n)$,
we deduce that
\begin{align*}
|x-y|^{\alpha-n}
\lesssim[\ell(Q)]^{\alpha-n}
=|Q|^{\frac{\alpha}{n}-1}\mathbf{1}_Q(x)\mathbf{1}_Q(y).
\end{align*}
This completes the proof of Lemma \ref{lem:kernel-dyadic}.
\end{proof}

In what follows, let the \emph{notation $L_{\rm c}^\infty$}
denote the set of all bounded measurable vector-valued functions
$\vec f:\mathbb R^n\to\mathbb C^d$ with compact support.
Let $\alpha$, $p$, and $q$ satisfy \eqref{eq:range}, and let
$W\in\mathscr A_{p,q}$. By the definition of
$I_\alpha$, the triangle inequality,
and Lemma \ref{lem:kernel-dyadic}, we find that,
for any $\vec f,\vec g\in L_{\rm c}^\infty$,
\begin{align}\label{eq:dyadic}
&\left|\int_{\mathbb R^n}
\left(W(x)I_\alpha\vec f(x),\vec g(x)\right)\,dx\right|\\
&\quad\leq
\int_{\mathbb R^n}\int_{\mathbb R^n}
\frac{\|W(x)W^{-1}(y)\|\,|W(y)\vec f(y)|\,|\vec g(x)|}
{|x-y|^{n-\alpha}}\,dy\,dx\nonumber\\
&\quad\lesssim
\sum_{\tau\in\{0,\frac13,\frac23\}^n}
\sum_{Q\in\mathcal{D}^\tau}|Q|^{\frac{\alpha}{n}-1}
\int_Q\int_Q\left\|W(x)W^{-1}(y)\right\|
\left|W(y)\vec f(y)\right|\,\left|\vec g(x)\right|\,dy\,dx\nonumber\\
&\quad=:\sum_{\tau\in\{0,\frac13,\frac23\}^n}
\Lambda_{\mathcal{D}^\tau}\left(W\vec f,\vec g\right).\nonumber
\end{align}
In what follows, the \emph{symbol $\mathcal{D}$} always denotes one fixed shifted
dyadic grid $\mathcal{D}^\tau$ in \eqref{eq:shifted-grid}
for some $\tau\in\{0,\frac13,\frac23\}^n$.
For any subfamily $\mathcal{F}\subset\mathcal{D}$
and any measurable vector-valued functions
$\vec f,\vec g:\mathbb R^n\to\mathbb C^d$, let
\begin{align*}
\Lambda_{\mathcal{F}}\left(\vec f,\vec g\right)
:=\sum_{Q\in\mathcal{F}}|Q|^{\frac{\alpha}{n}-1}
\int_Q\int_Q\left\|W(x)W^{-1}(y)\right\|
\left|\vec f(y)\right|\,\left|\vec g(x)\right|\,dy\,dx.
\end{align*}
Using \eqref{eq:dyadic}, we conclude that,
to prove the desired matrix-weighted estimate for $I_\alpha$
in Theorem \ref{thm:main}, it suffices to establish the
following estimate for $\Lambda_{\mathcal{D}}(\vec f,\vec g)$.

\begin{theorem}\label{thm:dyadic-form}
Let $\alpha$, $p$, and $q$ satisfy \eqref{eq:range}, and let
$W\in\mathscr A_{p,q}$. Assume that $\mathcal{D}$ is a shifted dyadic grid.
Then, for any $\vec f\in L^p$ and $\vec g\in L^{q'}$,
\begin{align*}
\Lambda_{\mathcal{D}}\left(\vec f,\vec g\right)
\lesssim[W]_{\mathscr A_{p,q}}^{(1-\frac{\alpha}{n})\max\{1,\frac{p'}{q}\}}
\left\|\vec f\right\|_{L^p}\left\|\vec g\right\|_{L^{q'}},
\end{align*}
where the implicit positive constant depends only on $n$, $d$,
$\alpha$, and $p$.
\end{theorem}

To prove Theorem \ref{thm:dyadic-form}, we need a decomposition
principle, which is inspired by some ideas from
the proof of \cite[Theorem 1.3]{l26}.
Suppose that $\mathcal{D}$ is a shifted dyadic grid.
Let $J$ be either $\{0,1,\dots,N\}$ for some $N\in\mathbb Z_+$
or $\mathbb Z_+$, and let $\mathcal{C}:=\{Q_j\}_{j\in J}$
be a sequence of cubes in $\mathcal{D}$.
We call $\mathcal{C}$ a \emph{chain in $\mathcal{D}$} if $Q_{j+1}\subsetneqq Q_j$
for $j,j+1\in J$. We call $Q_0$ the \emph{root} of $\mathcal{C}$ and write
$\mathcal{C}_{Q_0}:=\mathcal{C}$. Then we have the following
\emph{off-diagonal chain-packing principle of dyadic cubes}, which
is also of independent interest.

\begin{theorem}\label{thm-chain-packing}
Let $r,s\in[1,\infty]$ satisfy
$\Gamma:=\frac{1}{r}+\frac{1}{s}\in(1,\infty)$.
Let $\mathcal{D}$ be a shifted dyadic grid, and let
$\mathcal{F}\subset\mathcal{D}$ be a family of cubes
whose edge lengths are uniformly bounded above.
Then, for any $\vec{f}\in L^r$ and $\vec{g}\in L^{s}$, there exist a
family $\mathscr R\subset\mathcal{F}$ and chains $\{\mathcal{C}_R\}_{R\in\mathscr{R}}$
that are pairwise disjoint as collections of cubes such that
\begin{align}\label{eq:partition}
\mathcal{F}=\bigcup_{R\in\mathscr R}\mathcal{C}_R
\end{align}
and
\begin{align}\label{eq:packing}
\sum_{R\in\mathscr R}
\left\|\vec{f}\mathbf{1}_R\right\|_{L^r}
\left\|\vec{g}\mathbf{1}_R\right\|_{L^s}
\lesssim\left\|\vec{f}\right\|_{L^r}
\left\|\vec{g}\right\|_{L^s},
\end{align}
where the implicit positive constant depends only on $r$ and $s$
and is independent of $\vec{f}$, $\vec{g}$, and the upper bound for the
edge lengths of cubes in $\mathcal{F}$.
\end{theorem}

Before proving Theorem \ref{thm-chain-packing}, we explain how it
reduces the proof of Theorem \ref{thm:dyadic-form} to the single-chain case.
For any $N\in\mathbb N$, let $\mathcal{D}_N:=\{Q\in\mathcal{D}:\ell(Q)\le 2^N\}$.
By the monotone convergence theorem, it suffices to show
Theorem \ref{thm:dyadic-form} for $\mathcal{D}_N$ with the implicit positive
constant independent of $N$.
Assume that Theorem \ref{thm:dyadic-form}
has been established with the
shifted dyadic grids replaced by chains.
For any $N\in\mathbb{N}$, $\vec{f}\in L^p$, and $\vec{g}\in L^{q'}$,
applying Theorem \ref{thm-chain-packing} with $r:=p$, $s:=q'$, and
$\mathcal{F}:=\mathcal{D}_N$, we obtain a family $\mathscr R\subset\mathcal{D}_N$
and chains $\{\mathcal{C}_R\}_{R\in\mathscr{R}}$ such that
$\mathcal{D}_N=\bigcup_{R\in\mathscr R}\mathcal{C}_R$ and therefore
\begin{align*}
\Lambda_{\mathcal{D}_N}\left(\vec f,\vec g\right)
&=\sum_{R\in\mathscr R}
\Lambda_{\mathcal{C}_R}\left(\vec f\mathbf{1}_R,\vec g\mathbf{1}_R\right)\\
&\lesssim[W]_{\mathscr A_{p,q}}^{(1-\frac{\alpha}{n})\max\{1,\frac{p'}{q}\}}
\sum_{R\in\mathscr R}
\left\|\vec f\mathbf{1}_R\right\|_{L^p}\left\|\vec g\mathbf{1}_R\right\|_{L^{q'}}\\
&\lesssim[W]_{\mathscr A_{p,q}}^{(1-\frac{\alpha}{n})\max\{1,\frac{p'}{q}\}}
\left\|\vec f\right\|_{L^p}\left\|\vec g\right\|_{L^{q'}},
\end{align*}
and hence Theorem \ref{thm:dyadic-form} holds for $\mathcal{D}_N$
with the implicit positive constant independent of $N$.
Thus, Theorem \ref{thm-chain-packing} reduces the proof of
Theorem \ref{thm:dyadic-form} to the single-chain case.

We now prove Theorem \ref{thm-chain-packing}.
The proof of this theorem is inspired
by some ideas from the proof of Lerner \cite[Theorem 1.3]{l26}
and strongly depends on the assumption $\Gamma\in(1,\infty)$.

\begin{proof}[Proof of Theorem \ref{thm-chain-packing}]
If either $\|\vec f\|_{L^r}=0$ or
$\|\vec g\|_{L^s}=0$, it suffices to
take $\mathscr R:=\mathcal{F}$ and $\mathcal{C}_R:=\{R\}$ for any $R\in\mathcal{F}$.
Thus, in the remaining proof, we assume that both norms $\|\vec f\|_{L^r}$
and $\|\vec g\|_{L^s}$ are positive.
For any measurable set $E\subset\mathbb{R}^n$, let
\begin{align*}
\nu(E):=\frac{1}{\|\vec{f}\|_{L^r}^r}\int_E\left|\vec{f}(x)\right|^r\,dx
+\frac{1}{\|\vec{g}\|_{L^{s}}^{s}}\int_E\left|\vec{g}(x)\right|^{s}\,dx.
\end{align*}
Note that $\Gamma\in(1,\infty)$ implies $r,s\in [1,\infty)$,
and hence $\nu(E)$ is well defined. From its definition it follows 
that, for any $R\in\mathcal{F}$,
\begin{align*}
\frac{\|\vec{f}\mathbf{1}_R\|_{L^r}\|\vec{g}\mathbf{1}_R\|_{L^s}}
{\|\vec{f}\|_{L^r}\|\vec{g}\|_{L^s}}
\leq\left[\nu(R)\right]^{\frac{1}{r}}\left[\nu(R)\right]^{\frac{1}{s}}
=\left[\nu(R)\right]^\Gamma.
\end{align*}
Therefore, to show \eqref{eq:packing}, it suffices to construct a
partition into chains whose collection $\mathscr R$ of roots satisfies
\begin{align}\label{eq:mass-sum}
\sum_{R\in\mathscr R}\left[\nu(R)\right]^\Gamma\lesssim1.
\end{align}

We now state how to choose these chains.
For any $P,Q\in\mathcal{F}$ with $P\subsetneqq Q$, we call $P$ an
\emph{$\mathcal{F}$-child of $Q$} if there is no $R\in\mathcal{F}$ such that
$P\subsetneqq R\subsetneqq Q$.
For any $Q\in\mathcal{F}$, an $\mathcal{F}$-child $P$ of
$Q$ is called a \emph{continuation child} if
\begin{align}\label{eq:heavy}
\nu(P)>\frac{1}{2}\nu(Q).
\end{align}
Observe that the $\mathcal{F}$-children of any $Q\in\mathcal{F}$
are pairwise disjoint, and hence, for any $Q\in\mathcal{F}$,
there is at most one continuation child of $Q$.
Indeed, if there are two distinct continuation
children $P_1$ and $P_2$ of $Q$ satisfying \eqref{eq:heavy},
then $\nu(P_1)+\nu(P_2)>\nu (Q)$. However,
the disjointness of $P_1$ and $P_2$ yields
$\nu(P_1)+\nu(P_2)\leq\nu (Q)$, which is a contradiction.
All other $\mathcal{F}$-children of $Q$ are called \emph{restart children}.

Since $\mathcal{F}\subset\mathcal{D}$ and the edge lengths of cubes in $\mathcal{F}$
are uniformly bounded above, every cube in $\mathcal{F}$
is contained in a maximal cube of $\mathcal{F}$.
Let the \emph{notation $\mathscr R_0$} denote the collection of all maximal cubes in
$\mathcal{F}$. By Lemma \ref{lem:shifted-dyadic}(i), the members of $\mathscr R_0$ are
pairwise disjoint. For any $R\in\mathscr R_0$,
construct a chain with root $R$ as follows.
Let $Q_0:=R$. For any $j\in\mathbb{Z}_+$,
if $Q_j$ has a continuation child, denote it by
$Q_{j+1}$; otherwise stop. Since every cube has at most one
continuation child, this procedure produces a finite or
infinite chain, which we denote by the \emph{notation $\mathcal{C}_R$}.
Let the \emph{notation $\operatorname{Side}(R)$}
denote the collection of all restart children of cubes
in $\mathcal{C}_R$. Note that the cubes in $\operatorname{Side}(R)$ are pairwise disjoint.
Indeed, children of the same cube are disjoint, and a restart child
is disjoint from the continuation child and all its descendants.
Combining this disjointness and the fact that
every cube in $\operatorname{Side}(R)$ is contained in $R$, we conclude that,
for any $R\in \mathscr R_0$,
\begin{align}\label{eq:side-sum}
\sum_{P\in\operatorname{Side}(R)}\nu(P)
\leq\nu(R).
\end{align}
Moreover, if $P\in\operatorname{Side}(R)$ and $Q\in\mathcal{C}_R$ is its parent,
then
\begin{align}\label{eq:side-half}
\nu(P)
\leq
\frac12\nu(Q)
\leq
\frac12\nu(R).
\end{align}
From the assumption that $\Gamma>1$, \eqref{eq:side-sum}, and
\eqref{eq:side-half}, we deduce that, for any $R\in \mathscr R_0$,
\begin{align}\label{eq:contraction}
\sum_{P\in\operatorname{Side}(R)}\left[\nu(P)\right]^\Gamma\leq
\left[\frac{\nu(R)}{2}\right]^{\Gamma-1}
\sum_{P\in\operatorname{Side}(R)}\nu(P)
\leq2^{1-\Gamma}\left[\nu(R)\right]^\Gamma.
\end{align}
We next iterate this construction. To be precise,
for $m\in\mathbb{Z}_+$, let
\begin{align*}
\mathscr R_{m+1}:=
\bigcup_{R\in\mathscr R_m}\operatorname{Side}(R),
\end{align*}
and construct the chain $\mathcal{C}_R$ of continuation children for any
$R\in\mathscr R_{m+1}$ in the same way. Let
\begin{align*}
\mathscr R:=\bigcup_{m\in\mathbb{Z}_+}\mathscr R_m.
\end{align*}

\begin{figure}[H]
\centering
\begin{tikzpicture}[x=.18cm,y=.92cm,font=\small,line cap=round,
    cube/.style={line width=1.5pt},
    cubelabel/.style={font=\small}]

\draw[black,cube] (0,3.8)--(64,3.8)
    node[midway,above=3pt,cubelabel] {$R$};

\draw[myblue,cube] (0,2.8)--(16,2.8)
    node[midway,above=3pt,cubelabel] {$Q_1$};
\draw[myred,cube] (32,2.8)--(64,2.8)
    node[midway,above=3pt,cubelabel] {$R_1$};

\draw[myblue,cube] (0,1.6)--(4,1.6)
    node[midway,above=3pt,cubelabel] {$Q_2$};
\draw[myred,cube] (8,1.6)--(16,1.6)
    node[midway,above=3pt,cubelabel] {$R_2$};
\draw[myblue,cube] (32,1.6)--(40,1.6)
    node[midway,above=3pt,cubelabel] {$Q_3$};
\draw[myred,cube] (40,1.6)--(48,1.6)
    node[midway,above=3pt,cubelabel] {$R_3$};
\draw[myred,cube] (56,1.6)--(64,1.6)
    node[midway,above=3pt,cubelabel] {$R_4$};

\foreach \y in {.65,.87,1.09}
  \fill[black] (32,\y) circle[radius=1.2pt];

\draw[black,cube] (1,.05)--(4.5,.05);
\node[anchor=west,font=\small] at (6,.05) {root cube};
\draw[myblue,cube] (23,.05)--(26.5,.05);
\node[anchor=west,font=\small] at (28,.05)
    {continuation child};
\draw[myred,cube] (48,.05)--(51.5,.05);
\node[anchor=west,font=\small] at (53,.05)
    {restart child};
\end{tikzpicture}
\caption{All cubes are in ${\mathcal F}$
and, particularly, $R\in{\mathscr R}$. The black segment 
represents the initial root cube.
Blue and red segments respectively represent
continuation and restart children.
In this figure, $R_1,R_2\in\operatorname{Side}(R)$ and
$R_3,R_4\in\operatorname{Side}(R_1)$. Moreover,
$\{R,Q_1,Q_2,\cdots\}$ and $\{R_1,Q_3,\cdots\}$
are two chains.}
\end{figure}
We now verify that $\{\mathcal{C}_R\}_{R\in\mathscr{R}}$ is
the desired partition in \eqref{eq:partition}.
The inclusion $\bigcup_{R\in\mathscr R}\mathcal{C}_R\subset\mathcal{F}$
follows from the construction. To show the reverse inclusion,
let $Q\in\mathcal{F}$.
Since the edge lengths of cubes in $\mathcal{F}$ are uniformly
bounded above, $Q$ is connected to a unique maximal
cube in $\mathscr R_0$ by a finite path in $\mathcal{F}$. Starting from
this maximal cube and moving downward along the path, each child
either is the continuation child of its parent and hence remains in
the same chain, or is a restart child and hence becomes the root of a
new chain belonging to $\mathscr R$. Thus, every cube in $\mathcal{F}$ belongs to $\mathcal{C}_R$ for some
$R\in\mathscr R$. Moreover, this construction also guarantees that each cube
belongs to a unique chain chosen above.
Therefore, \eqref{eq:partition} holds and the
chains are pairwise disjoint as collections of cubes.

Finally, we verify that $\mathscr{R}$ satisfies \eqref{eq:mass-sum}. By the facts that
the members of $\mathscr R_0$ are pairwise disjoint and
that $\Gamma\in(1,\infty)$, we find that
\begin{align*}
\sum_{R\in\mathscr R_0}\left[\nu(R)\right]^\Gamma
\leq\left[\sum_{R\in\mathscr R_0}\nu(R)\right]^\Gamma
\leq \left[\nu(\mathbb R^n)\right]^\Gamma=2^\Gamma.
\end{align*}
Applying \eqref{eq:contraction} to each root in
the $m$-th generation, we conclude that, for any $m\in\mathbb{Z}_+$,
\begin{align*}
\sum_{R\in\mathscr R_{m+1}}\left[\nu(R)\right]^\Gamma
\leq2^{1-\Gamma}\sum_{R\in\mathscr R_m}\left[\nu(R)\right]^\Gamma,
\end{align*}
and hence
\begin{align*}
\sum_{R\in\mathscr R}\left[\nu(R)\right]^\Gamma=
\sum_{m\in\mathbb{Z}_+}\sum_{R\in\mathscr R_m}
\left[\nu(R)\right]^\Gamma\leq2^\Gamma\sum_{m\in\mathbb{Z}_+}
2^{m(1-\Gamma)}=\frac{2^\Gamma}{1-2^{1-\Gamma}}<\infty,
\end{align*}
where the last inequality follows from the fact that $\Gamma\in(1,\infty)$.
Therefore, \eqref{eq:mass-sum} holds, which,
by the argument at the beginning of the present proof,
further implies \eqref{eq:packing}.
This completes the proof of Theorem \ref{thm-chain-packing}.
\end{proof}

\begin{remark}\label{rek-chain-endpoint}
Note that the above proof strongly relies on the assumption $\Gamma\in(1,\infty)$.
Moreover, when $\Gamma\in[0,1]$, Theorem \ref{thm-chain-packing} fails,
and hence the assumption $\Gamma\in(1,\infty)$ in this theorem is sharp.
To prove this, fix $r,s\in[1,\infty]$
satisfying $\Gamma:=\frac1r+\frac1s\in[0,1]$.
Next, let $N\in\mathbb N$ and $\mathcal F:=\{Q\in\mathcal{D}:\ell(Q)\le 2^N\}$.
Take $Q_0\in\mathcal F$ with $\ell(Q_0)=2^N$ and take
$\vec f=\vec g:=\mathbf1_{Q_0}\vec e$, where $\vec e\in\mathbb C^d$
is a vector with $|\vec{e}|=1$.
If the conclusion of Theorem \ref{thm-chain-packing} also holds
for such $\Gamma$ and $\mathcal{F}$, then $\mathcal{F}$ can be partitioned
into chains $\{\mathcal C_R\}_{R\in\mathscr R}$ as in \eqref{eq:partition}
with $\mathscr R\subset\mathcal{F}$, and \eqref{eq:packing} would hold.
Since $Q_0$ is a maximal cube of $\mathcal F$, it follows that,
for any $R\in\mathscr R$, if $\mathcal C_R$ contains a cube $Q\subset Q_0$,
then $R\subset Q_0$. We next consider two cases for $\Gamma$.
If $\Gamma\in(0,1]$, then, in this case,
for any $R\in\mathscr R$ with $R\subset Q_0$,
\begin{align}\label{eq-Q-R}
\sum_{Q\in\mathcal C_R}|Q|^\Gamma
\leq |R|^\Gamma\sum_{j\in\mathbb{Z}_+}2^{-nj\Gamma}
=\frac{|R|^\Gamma}{1-2^{-n\Gamma}}.
\end{align}
On the other hand, summing over all the dyadic cubes contained in $Q_0$,
we obtain
\begin{align*}
\sum_{\genfrac{}{}{0pt}{}{Q\in\mathcal D_N}{Q\subset Q_0}}
|Q|^\Gamma=|Q_0|^\Gamma\sum_{j\in\mathbb{Z}_+}2^{nj(1-\Gamma)}
=\infty.
\end{align*}
Therefore, by this and \eqref{eq-Q-R}, we find that
\begin{align*}
\infty=\sum_{\genfrac{}{}{0pt}{}{Q\in\mathcal D_N}{Q\subset Q_0}}
|Q|^\Gamma=\sum_{\genfrac{}{}{0pt}{}{R\in\mathscr R}{R\subset Q_0}}
\sum_{Q\in\mathcal C_R}|Q|^\Gamma\lesssim
\sum_{\genfrac{}{}{0pt}{}{R\in\mathscr R}{R\subset Q_0}}|R|^\Gamma
\leq\sum_{R\in\mathscr R}
\left\|\vec f\mathbf1_R\right\|_{L^r}
\left\|\vec g\mathbf1_R\right\|_{L^s},
\end{align*}
where, in the last inequality, we used the fact that
$\|\vec f\mathbf1_R\|_{L^r}\|\vec g\mathbf1_R\|_{L^s}
=|R|^{\frac{1}{r}+\frac{1}{s}}=|R|^\Gamma$.
However, $\|\vec f\|_{L^r}\|\vec g\|_{L^s}=|Q_0|^\Gamma<\infty$,
which contradicts \eqref{eq:packing}.

If $\Gamma=0$, then $r=s=\infty$ in this case.
Observe that, for any $j\in\mathbb{N}$,
there are $2^{nj}$ mutually disjoint cubes
contained in $Q_0$ with volume $2^{-jn}|Q_0|$.
Each of these cubes belongs to a unique chain $\mathcal C_R$
whose root $R$ is contained in $Q_0$.
Since a chain contains at most one cube of a given edge length,
these $2^{nj}$ cubes belong to distinct chains.
Thus, there are at least $2^{nj}$ chains with roots contained in $Q_0$,
which, together with the arbitrariness of $j$, further implies that
there are infinitely many roots of chains contained in $Q_0$.
Therefore,
\begin{align*}
\infty=\sum_{\genfrac{}{}{0pt}{}{R\in\mathscr R}{R\subset Q_0}}1
=\sum_{\genfrac{}{}{0pt}{}{R\in\mathscr R}{R\subset Q_0}}
\left\|\vec f\mathbf1_R\right\|_{L^\infty}
\left\|\vec g\mathbf1_R\right\|_{L^\infty}
\leq\sum_{R\in\mathscr R}\left\|\vec f\mathbf1_R\right\|_{L^\infty}
\|\vec g\mathbf1_R\|_{L^\infty}.
\end{align*}
Note that $\|\vec f\|_{L^\infty}\|\vec g\|_{L^\infty}=1$,
which again contradicts \eqref{eq:packing}.
Therefore, the assumption $\Gamma\in(1,\infty)$
in Theorem \ref{thm-chain-packing} is necessary and hence sharp.
\end{remark}

\subsection{Single-Chain Estimates and Proof of Theorem \ref{thm:main}}\label{sec:chain-estimates}

In this subsection, we first prove Theorem \ref{thm:dyadic-form}
with the shifted dyadic grids replaced by chains
(see Lemma \ref{lem-simple-chain}). The proof of this chain estimate
is based on an idea of Nazarov et al. \cite[Lemma 5.7]{nptv17}.
Combining this estimate with the
off-diagonal chain-packing principle of dyadic cubes
established in the previous subsection,
we complete the proof of Theorem \ref{thm:dyadic-form},
and finally show Theorem \ref{thm:main}.

To obtain the chain estimate, we need the following two lemmas.
Let $\mathcal{D}$ be a shifted dyadic grid, and let $\sigma\in[0,n)$. The
\emph{dyadic fractional Hardy--Littlewood maximal operator}
$M_\sigma^{\mathcal{D}}$ is defined by setting,
for any measurable function
$f$ on $\mathbb{R}^n$ and for any $x\in \mathbb{R}^n$,
\begin{align*}
M_\sigma^{\mathcal{D}}(f)(x)
:=\sup_{Q\in\mathcal{D}:\,x\in Q}|Q|^{\frac{\sigma}{n}}\fint_Q|f(y)|\,dy.
\end{align*}
In particular, let $M^{\mathcal{D}}:=M_0^{\mathcal{D}}$. By the well-known
boundedness of the dyadic Hardy--Littlewood maximal operator
(see, for instance, \cite[Theorem 15.1]{ln19}),
we find that, for any $r\in(1,\infty)$ and $f\in L^r$,
\begin{align}\label{eq:maximal}
\left\|M^{\mathcal{D}}(f)\right\|_{L^r}\le r'\|f\|_{L^r}.
\end{align}
Using \eqref{eq:maximal}, we obtain the following
boundedness of fractional operators $M_\sigma^{\mathcal{D}}$.

\begin{lemma}\label{lem:frac-max}
Let $\mathcal{D}$ be a shifted dyadic grid.
Let $\sigma\in(0,n)$, $r\in(1,\frac{n}{\sigma})$, and
$s\in(r,\infty)$ satisfy $\frac{1}{s}=\frac{1}{r}-\frac{\sigma}{n}$.
Then, for any $f\in L^r$,
\begin{align}\label{eq:frac-max}
\left\|M_\sigma^{\mathcal{D}}(f)\right\|_{L^s}
\leq\left(r'\right)^{\frac{r}{s}}\left\|f\right\|_{L^r}.
\end{align}
\end{lemma}

\begin{proof}
Let $f\in L^r$ and $\theta:=\frac{\sigma r}{n}$.
By H\"older's inequality, we find that,
for any cube $Q\in\mathcal{D}$,
\begin{align*}
|Q|^{\frac{\sigma}{n}}\fint_Q|f(y)|\,dy
\le\|f\|_{L^r}^{\theta}
\left[\fint_Q|f(y)|\,dy\right]^{1-\theta}.
\end{align*}
Noting $1-\theta=\frac{r}{s}$, for any
$x\in\mathbb{R}^n$, taking the supremum of both sides over
all cubes in $\mathcal{D}$ containing $x$ yields
\begin{align*}
M_\sigma^{\mathcal{D}}(f)(x)
\le\|f\|_{L^r}^{1-\frac{r}{s}}
\left[M^{\mathcal{D}}(f)(x)\right]^{\frac{r}{s}}.
\end{align*}
This, together with \eqref{eq:maximal},
further implies \eqref{eq:frac-max}. This completes the proof
of Lemma \ref{lem:frac-max}.
\end{proof}

To present the next lemma, we first recall the concept of sparse families
(see, for instance, \cite[Definition 2.2]{nptv17}).
\begin{definition}
Let $\mathcal{D}$ be a shifted dyadic grid in $\mathbb R^n$, and let
$\eta\in(0,1)$. A family $\mathcal{S}\subset\mathcal{D}$ is said to be
\emph{$\eta$-sparse} if, for any $Q\in\mathcal{S}$, there exists a
measurable set $E_Q\subset Q$ such that $|E_Q|\geq\eta|Q|$
and the sets $\{E_Q\}_{Q\in\mathcal{S}}$ are pairwise disjoint.
Any such family $\{E_Q\}_{Q\in\mathcal{S}}$ is called a family of
\emph{major subsets associated with $\mathcal{S}$}.
\end{definition}

Combining the fractional maximal estimate in Lemma \ref{lem:frac-max}
with the reverse H\"{o}lder estimate in Lemma \ref{lem:op-RH},
we obtain the following off-diagonal Carleson-type embedding result.

\begin{lemma}\label{lem:off-diagonal-Carleson}

Let $\alpha$, $p$, and $q$ satisfy \eqref{eq:range}, and let
$W\in\mathscr A_{p,q}$. Suppose that $\mathcal{D}$
is a shifted dyadic grid
and $\mathcal{F}\subset\mathcal{D}$ is an $\eta$-sparse family for some
$\eta\in(0,1)$. Let $\{A_Q\}_{Q\in\mathcal{F}}$ be a family of
reducing operators of order $q$ for $W$.
Then, for any $g\in L^{q'}$,
\begin{align}\label{eq:embedding}
\left\{\sum_{Q\in\mathcal{F}}\left[
|Q|^{-\frac1q}\int_Q \left\|W(x)A_Q^{-1}\right\||g(x)|\,dx
\right]^{p'}\right\}^{\frac1{p'}}
\lesssim[W]_{\mathscr A_{p,q}}^{\frac{1}{p'}}\|g\|_{L^{q'}},
\end{align}
where the implicit positive constant depends
only on $n$, $d$, $p$, $q$, and $\eta$.
\end{lemma}

\begin{proof}
By the monotone convergence theorem, we find that
it suffices to prove \eqref{eq:embedding} for a finite
family $\mathcal{F}$ with the implicit positive constant independent of its cardinality.
Let $\varepsilon_W$ be as in \eqref{eq:epsilon}, and let
$s:=q(1+\varepsilon_W)$.
For any $Q\in \mathcal{F}$ and $x\in \mathbb{R}^n$,
let $a_Q(x):=\|W(x)A_Q^{-1}\|$.
Observe that $s'\in(1,q')$.
Applying Lemma \ref{lem:op-RH} and \eqref{eq:reduce},
we conclude that, for any $Q\in\mathcal{F}$,
\begin{align*}
\left(\fint_Q[a_Q(x)]^s\,dx\right)^{\frac1s}\lesssim1.
\end{align*}
Fix $g\in L^{q'}$.
From this and H\"older's inequality, we deduce that,
for any $Q\in\mathcal{F}$,
\begin{align}\label{eq:local-pair}
|Q|^{-\frac1q}\int_Qa_Q(x)|g(x)|\,dx\lesssim
|Q|^{\frac1{q'}}\left[\fint_Q|g(x)|^{s'}\,dx\right]^{\frac1{s'}}.
\end{align}
Note that $\frac{p'}{q'}=1+\frac{\alpha p'}{n}$.
Therefore, taking the $p'$-power of both sides of
\eqref{eq:local-pair} and
summing over all $Q\in\mathcal{F}$, we obtain
\begin{align}\label{eq:frac-sum}
&\sum_{Q\in\mathcal{F}}\left[
|Q|^{-\frac1q}\int_Qa_Q(x)|g(x)|\,dx\right]^{p'}\\
&\qquad\lesssim\sum_{Q\in\mathcal{F}}|Q|
\left[|Q|^{\frac{\alpha s'}{n}}
\fint_Q|g(x)|^{s'}\,dx\right]^{\frac{p'}{s'}}.\nonumber
\end{align}
We now estimate the right-hand side of \eqref{eq:frac-sum}.
From the definition of $\eta$-sparse families, we infer that
there exists a family $\{E_Q\}_{Q\in\mathcal{F}}$ of pairwise disjoint sets such that,
for any $Q\in\mathcal{F}$, $E_Q\subset Q$ and $|E_Q|\geq\eta|Q|$.
Using the properties of $\{E_Q\}_{Q\in\mathcal{F}}$ and the definition of
$M_{\alpha s'}^{\mathcal{D}}$, we conclude that, for any
$Q\in\mathcal{F}$ and $x\in E_Q$,
\begin{align*}
|Q|^{\frac{\alpha s'}{n}}
\fint_Q|g(y)|^{s'}\,dy
\leq M_{\alpha s'}^{\mathcal{D}}(|g|^{s'})(x),
\end{align*}
and hence
\begin{align}\label{eq:max-bound}
\sum_{Q\in\mathcal{F}}\left[
|Q|^{-\frac1q}\int_Qa_Q(x)|g(x)|\,dx
\right]^{p'}\lesssim
\left\|M_{\alpha {s'}}^{\mathcal{D}}(|g|^{s'})
\right\|_{L^{\frac{p'}{s'}}}^{\frac{p'}{s'}}.
\end{align}
Let $r:=\frac{q'}{s'}$ and $u:=\frac{p'}{s'}$.
Then $1<r<u$ and $\frac1r-\frac1u=\frac{\alpha {s'}}{n}$.
In particular, $r\in(1,\frac{n}{\alpha {s'}})$. Noting
$u{s'}=p'$ and $r{s'}=q'$, applying Lemma \ref{lem:frac-max} to
\eqref{eq:max-bound}, we obtain
\begin{align}\label{eq:seq-bound}
\left\{\sum_{Q\in\mathcal{F}}\left[
|Q|^{-\frac1q}\int_Qa_Q(x)|g(x)|\,dx
\right]^{p'}\right\}^{\frac1{p'}}
&\lesssim\left[(r')^{\frac r u}
\left\||g|^{s'}\right\|_{L^r}
\right]^{\frac1{s'}}\\
&=(r')^{\frac r{p'}}\|g\|_{L^{q'}}.\nonumber
\end{align}
By the definitions of $s$ and $r$, we find that
$r=1+\frac{\varepsilon_W}
{(q-1)(1+\varepsilon_W)}$.
Thus, the fact that $\varepsilon_W\in(0,1]$ yields
\begin{align*}
r'=q+\frac{q-1}{\varepsilon_W}
\sim\varepsilon_W^{-1}\text{ and }
r-1=\frac{\varepsilon_W}
{(q-1)(1+\varepsilon_W)}\sim\varepsilon_W.
\end{align*}
Using the fact that
$\sup_{v\in(0,1]}v^{-\lambda v}<\infty$ for any
$\lambda\in(0,\infty)$ and using the definition
\eqref{eq:epsilon} of $\varepsilon_W$, we conclude that
\begin{align}\label{eq:RH-factor}
(r')^{\frac r{p'}}=
(r')^{\frac1{p'}}(r')^{\frac{r-1}{p'}}
\lesssim\varepsilon_W^{-\frac1{p'}}
\varepsilon_W^{-C\varepsilon_W}
\lesssim\varepsilon_W^{-\frac1{p'}}
\sim[W]_{\mathscr A_{p,q}}^{\frac{1}{p'}},
\end{align}
where the positive constant $C$ depends
only on $q$. Combining \eqref{eq:seq-bound} and
\eqref{eq:RH-factor}, we obtain \eqref{eq:embedding}
for finite families with the implicit positive
constant independent of their cardinalities.
This completes the proof of Lemma \ref{lem:off-diagonal-Carleson}.
\end{proof}

Using the decomposition in the proof of \cite[Lemma 5.7]{nptv17},
we can now show the following chain estimate.

\begin{lemma}\label{lem-simple-chain}
Let $\alpha$, $p$, and $q$ satisfy \eqref{eq:range}, and let
$W\in\mathscr A_{p,q}$. Suppose that $\mathcal{D}$ is a shifted dyadic grid
in $\mathbb R^n$ and $\mathcal{C}$ is a chain in $\mathcal{D}$.
Then, for any $\vec{f}\in L^p$ and $\vec{g}\in L^{q'}$,
\begin{align}\label{eq:chain}
\Lambda_{\mathcal{C}}\left(\vec{f},\vec{g}\right)\lesssim
[W]_{\mathscr A_{p,q}}^{(1-\frac{\alpha}{n})\max\{1,\frac{p'}{q}\}}
\left\|\vec{f}\right\|_{L^p}\left\|\vec{g}\right\|_{L^{q'}},
\end{align}
where the implicit positive constant depends only on $n$, $d$, $\alpha$,
and $p$.
\end{lemma}
\begin{proof}
By the monotone convergence theorem, we find that it suffices to prove
\eqref{eq:chain} for finite chains with the implicit positive
constant independent of the number of cubes.
Suppose that $N\in\mathbb Z_+$ and
\begin{align*}
\mathcal{C}=\{Q_0,Q_1,\dots,Q_N\},
\end{align*}
where $Q_N\subsetneqq\cdots\subsetneqq Q_1\subsetneqq Q_0$.
For any $Q\in\mathcal{C}$, let $A_Q$ be a reducing operator of order $q$
for $W$. For any $j\in\{0,\dots,N\}$, let
\begin{align*}
E_{Q_j}:=
\begin{cases}
Q_N&\text{if }j=N,\\
{Q_j}\setminus Q_{j+1}&\text{otherwise}.
\end{cases}
\end{align*}
It follows from this definition and the inclusions
$Q_N\subsetneqq\cdots\subsetneqq Q_1\subsetneqq Q_0$ that
the sets $\{E_Q\}_{Q\in\mathcal{C}}$ are pairwise disjoint and satisfy,
for any $Q\in\mathcal{C}$,
\begin{align}\label{eq:chain-sets}
\left(1-\frac{1}{2^n}\right)|Q|\leq|E_Q|\leq|Q|.
\end{align}
For any $Q\in\mathcal{C}\setminus\{Q_N\}$, let
$Q'$ denote its child in $\mathcal{C}$.
Then $Q\times Q$ can be decomposed into
the following pairwise disjoint sets:
\begin{align*}
Q\times Q=
(Q\times E_Q)
\cup
(E_Q\times Q')
\cup
(Q'\times Q').
\end{align*}
Note that $E_{Q_N}=Q_N$, and hence
$Q_N\times Q_N=Q_N\times E_{Q_N}$.
Applying this decomposition to each term of
$\Lambda_{\mathcal{C}}$, we conclude that,
for any $\vec{f}\in L^p$ and $\vec{g}\in L^{q'}$,
\begin{align}\label{eq:split}
\Lambda_{\mathcal{C}}\left(\vec{f},\vec{g}\right)
&=\sum_{Q\in\mathcal{C}}|Q|^{\frac{\alpha}{n}-1}\int_Q\int_Q
\left\|W(x)W^{-1}(y)\right\|\left|\vec{f}(y)\right|
\left|\vec{g}(x)\right|\,dy\,dx\\
&=\sum_{Q\in\mathcal{C}}|Q|^{\frac{\alpha}{n}-1}\int_Q\int_{E_Q}\cdots\,dy\,dx
+\sum_{\genfrac{}{}{0pt}{}{Q\in\mathcal{C}}{Q\neq Q_N}}
|Q|^{\frac{\alpha}{n}-1}\int_{E_Q}\int_{Q'}\cdots\,dy\,dx\nonumber\\
&\quad+\sum_{\genfrac{}{}{0pt}{}{Q\in\mathcal{C}}{Q\neq Q_N}}
|Q|^{\frac{\alpha}{n}-1}\int_{Q'}\int_{Q'}\cdots\,dy\,dx\nonumber\\
&=:\Lambda_1+\Lambda_2+\Lambda_3.\nonumber
\end{align}
Observe that, for any $Q\in\mathcal{C}\setminus\{Q_N\}$,
\begin{align*}
|Q|^{\frac{\alpha}{n}-1}=
\left(\frac{|Q'|}{|Q|}\right)^{1-\frac{\alpha}{n}}
|Q'|^{\frac{\alpha}{n}-1}\leq2^{\alpha-n}|Q'|^{\frac{\alpha}{n}-1},
\end{align*}
where the inequality follows from the assumption that $\alpha\in(0,n)$
and the fact that $Q'\subsetneqq Q$ and both are dyadic cubes.
Using this, we obtain $\Lambda_3\le2^{\alpha-n}\Lambda_{\mathcal{C}}$,
and hence
\begin{align}\label{eq:absorb}
\left(1-2^{\alpha-n}\right)
\Lambda_{\mathcal{C}}\le\Lambda_1+\Lambda_2.
\end{align}
Thus, to prove \eqref{eq:chain},
it suffices to estimate $\Lambda_1$ and $\Lambda_2$ in \eqref{eq:split}.

We first estimate $\Lambda_1$.
Note that
\begin{align}\label{eq:edge-form}
\Lambda_1\leq
\sum_{Q\in\mathcal{C}}|Q|^{\frac{\alpha}{n}-1}
\int_Q\left\|W(x)A_Q^{-1}\right\|
\left|\vec g(x)\right|\,dx
\int_{E_Q}\left\|A_QW^{-1}(y)\right\|
\left|\vec f(y)\right|\,dy.
\end{align}
By H\"older's inequality and \eqref{eq:dual-reduce},
for any $Q\in\mathcal{C}$,
\begin{align}\label{eq:edge-bound}
\int_{E_Q}\left\|A_QW^{-1}(y)\right\|\left|\vec f(y)\right|\,dy
\lesssim[W]_{\mathscr A_{p,q}}^{\frac{1}{q}}|Q|^{\frac{1}{p'}}\left\|\vec f\mathbf{1}_{E_Q}\right\|_{L^p}.
\end{align}
Noting $\frac{\alpha}{n}+\frac{1}{p'}=\frac{1}{q'}$,
combining \eqref{eq:edge-form} and \eqref{eq:edge-bound} yields
\begin{align*}
\Lambda_1\lesssim[W]_{\mathscr A_{p,q}}^{\frac{1}{q}}\sum_{Q\in\mathcal{C}}
\left\|\vec f\mathbf{1}_{E_Q}\right\|_{L^p}
|Q|^{-\frac1q}\int_Q\left\|W(x)A_Q^{-1}\right\|
|\vec g(x)|\,dx.
\end{align*}
From \eqref{eq:chain-sets}, we deduce that
$\mathcal{C}$ is $(1-2^{-n})$-sparse. Thus, applying H\"older's inequality
and Lemma \ref{lem:off-diagonal-Carleson}, we conclude that
\begin{align*}
\Lambda_1
&\lesssim[W]_{\mathscr A_{p,q}}^{\frac{1}{q}}
\left(\sum_{Q\in\mathcal{C}}\left\|\vec f\mathbf{1}_{E_Q}\right\|_{L^p}^p\right)^{\frac1p}
\left\{\sum_{Q\in\mathcal{C}}\left[
|Q|^{-\frac1q}\int_Q\left\|W(x)A_Q^{-1}\right\|
|\vec g(x)|\,dx\right]^{p'}\right\}^{\frac1{p'}}\\
&\lesssim[W]_{\mathscr A_{p,q}}^{\frac{1}{q}+\frac{1}{p'}}
\left(\sum_{Q\in\mathcal{C}}\left\|\vec f\mathbf{1}_{E_Q}
\right\|_{L^p}^p\right)^{\frac1p}
\|\vec g\|_{L^{q'}}.
\end{align*}
This, together with the fact that $\{E_Q\}_{Q\in\mathcal{C}}$
are pairwise disjoint, further implies that
\begin{align}\label{eq:edge-one}
\Lambda_1\lesssim[W]_{\mathscr A_{p,q}}^{\frac{1}{q}+\frac{1}{p'}}
\left\|\vec f\right\|_{L^p}\|\vec g\|_{L^{q'}},
\end{align}
which is the desired estimate for $\Lambda_1$.

We next estimate $\Lambda_2$ by duality.
To this end, let $V:=W^{-1}$. Note that, for any
$Q\in\mathcal{C}\setminus\{Q_N\}$, $Q'\subset Q$.
Thus, it follows from Tonelli's theorem and
Lemma \ref{lem-norm-AB=BA} that
\begin{align*}
\Lambda_2&\leq
\sum_{\genfrac{}{}{0pt}{}{Q\in\mathcal{C}}{Q\neq Q_N}}
|Q|^{\frac{\alpha}{n}-1}\int_{E_Q}\int_Q
\left\|W(x)W^{-1}(y)\right\|\left|\vec{f}(y)\right|
\left|\vec{g}(x)\right|\,dy\,dx\\
&=\sum_{\genfrac{}{}{0pt}{}{Q\in\mathcal{C}}{Q\neq Q_N}}
|Q|^{\frac{\alpha}{n}-1}\int_Q\int_{E_Q}
\left\|V(x)V^{-1}(y)\right\|\left|\vec{g}(y)\right|
\left|\vec{f}(x)\right|\,dy\,dx.
\end{align*}
By Remark \ref{rmk-dual}, we find that
$V\in\mathscr A_{q',p'}$ and
$[V]_{\mathscr A_{q',p'}}^{\frac{1}{p'}}
\sim[W]_{\mathscr A_{p,q}}^{\frac{1}{q}}$.
Observe that
$\frac{1}{q'}-\frac{1}{p'}=\frac{\alpha}{n}$.
Using these observations and applying \eqref{eq:edge-one} with
$(p,q,W,\vec{f},\vec{g})$ replaced by
$(q',p',V,\vec{g},\vec{f})$, we obtain
\begin{align}\label{eq:edge-two}
\Lambda_2
\lesssim [V]_{\mathscr A_{q',p'}}^{
\frac{1}{p'}+\frac{1}{q}}
\left\|\vec{f}\right\|_{L^p}
\left\|\vec{g}\right\|_{L^{q'}}
\lesssim[W]_{\mathscr A_{p,q}}^{\frac{p'}{q}
(\frac{1}{p'}+\frac{1}{q})}
\left\|\vec{f}\right\|_{L^p}
\left\|\vec{g}\right\|_{L^{q'}}.
\end{align}
Recall that $[W]_{\mathscr A_{p,q}}\geq 1$ [see \eqref{eq:Apq-lower}].
Combining \eqref{eq:absorb}, \eqref{eq:edge-one}, and
\eqref{eq:edge-two}, and using
$1-\frac{\alpha}{n}=\frac{1}{p'}+\frac{1}{q}$, we conclude that
\begin{align*}
\Lambda_{\mathcal{C}}\left(\vec{f},\vec{g}\right)\lesssim
[W]_{\mathscr A_{p,q}}^{(1-\frac{\alpha}{n})
\max\{1,\frac{p'}{q}\}}
\left\|\vec{f}\right\|_{L^p}
\left\|\vec{g}\right\|_{L^{q'}}.
\end{align*}
Thus, \eqref{eq:chain} holds.
This completes the proof of Lemma \ref{lem-simple-chain}.
\end{proof}

Based on Lemma \ref{lem-simple-chain} and
Theorem \ref{thm-chain-packing},
we now prove Theorem \ref{thm:dyadic-form}.

\begin{proof}[Proof of Theorem \ref{thm:dyadic-form}]
For $N\in\mathbb N$, let $\mathcal{D}_N:=
\{Q\in\mathcal{D}:\,\ell(Q)\leq 2^N\}$.
Fix $\vec{f}\in L^p$ and $\vec{g}\in L^{q'}$.
For any $N\in\mathbb N$,
applying Theorem \ref{thm-chain-packing} with
$r:=p$ and $s:=q'$ to
$\vec{f}$, $\vec{g}$, and the family
$\mathcal{D}_N$, we obtain a family
$\mathscr R_N$ of cubes in $\mathcal{D}_N$ such that
$\mathcal{D}_N=\bigcup_{R\in\mathscr R_N}\mathcal{C}_R$,
where the chains $\mathcal{C}_R$ are pairwise disjoint as collections
of cubes and each has root $R$.
Since, for any $R\in\mathscr R_N$,
every cube in $\mathcal{C}_R$ is contained in $R$, we deduce that
\begin{align*}
\Lambda_{\mathcal{C}_R}\left(\vec{f},\vec{g}\right)
=\Lambda_{\mathcal{C}_R}\left(\vec{f}\mathbf{1}_R,\vec{g}\mathbf{1}_R\right).
\end{align*}
Using this identity, Lemma \ref{lem-simple-chain}, and
\eqref{eq:packing}, we conclude that, for any $N\in\mathbb N$,
\begin{align*}
\Lambda_{\mathcal{D}_N}\left(\vec{f},\vec{g}\right)
&=\sum_{R\in\mathscr R_N}
\Lambda_{\mathcal{C}_R}\left(\vec{f},\vec{g}\right)
=\sum_{R\in\mathscr R_N}
\Lambda_{\mathcal{C}_R}\left(\vec{f}\mathbf{1}_R,\vec{g}\mathbf{1}_R\right)\\
&\lesssim[W]_{\mathscr A_{p,q}}^{(1-\frac{\alpha}{n})\max\{1,\frac{p'}{q}\}}
\sum_{R\in\mathscr R_N}
\left\|\vec{f}\mathbf{1}_R\right\|_{L^p}
\left\|\vec{g}\mathbf{1}_R\right\|_{L^{q'}}\\
&\lesssim[W]_{\mathscr A_{p,q}}^{(1-\frac{\alpha}{n})\max\{1,\frac{p'}{q}\}}
\left\|\vec{f}\right\|_{L^p}
\left\|\vec{g}\right\|_{L^{q'}},
\end{align*}
where the implicit positive constants are
independent of $N$. Finally, letting $N\to\infty$ and using
the monotone convergence theorem, we obtain the
desired estimate for $\mathcal{D}$.
This completes the proof of Theorem \ref{thm:dyadic-form}.
\end{proof}

Finally, we show Theorem \ref{thm:main}.

\begin{proof}[Proof of Theorem \ref{thm:main}]
We first prove the estimate for $\vec f\in L_{\rm c}^\infty$.
By \eqref{eq:dyadic} and Theorem \ref{thm:dyadic-form}, for any
$\vec g\in L_{\rm c}^\infty$,
\begin{align*}
\left|\int_{\mathbb R^n}
\left(W(x)I_\alpha\vec f(x),\vec g(x)\right)\,dx\right|
\lesssim
[W]_{\mathscr A_{p,q}}^{(1-\frac{\alpha}{n})\max\{1,\frac{p'}{q}\}}
\|W\vec f\|_{L^p}\|\vec g\|_{L^{q'}}.
\end{align*}
Noting $L_{\rm c}^\infty$ is dense in $L^{q'}$, duality yields
$I_\alpha \vec f\in L^q(W^q)$ and
\begin{align}\label{eq:bounded}
\left\|I_\alpha \vec{f}\right\|_{L^q(W^q)}\lesssim
[W]_{\mathscr A_{p,q}}^{(1-\frac{\alpha}{n})\max\{1,\frac{p'}{q}\}}
\left\|\vec{f}\right\|_{L^p(W^p)}.
\end{align}

Next, we show \eqref{eq:bounded}
for general $\vec f\in L^p(W^p)$ and hence
\eqref{eq:main} holds. To this end, we
first prove that, for almost every $x\in\mathbb R^n$,
\begin{align}\label{eq:absolute}
\int_{\mathbb{R}^n}
\frac{|\vec{f}(y)|}{|x-y|^{n-\alpha}} \,dy<\infty,
\end{align}
and hence $I_\alpha \vec{f}$ is well defined.
For any $m\in\mathbb{N}$, let
$\vec{f}_m:=\vec{f}\mathbf{1}_{E_m}$, where
\begin{align*}
E_m:=\left\{x\in\mathbb{R}^n:\,|x|<m\right\}\cap\left\{
x\in\mathbb{R}^n:\,\left|\vec{f}(x)\right|\leq m\right\}.
\end{align*}
Note that, for any $m\in\mathbb{N}$,
$\vec{f}_m\in L_{\rm c}^\infty$.
For any $j\in\{1,\dots,d\}$, let $f^{(j)}$ denote the
$j$-th component function of $\vec f$ and, for any
$x\in\mathbb{R}^n$, let
\begin{align*}
\sigma_j(x):=
\begin{cases}
\displaystyle\frac{\overline{f^{(j)}(x)}}{|f^{(j)}(x)|}&\text{if }f^{(j)}(x)\neq0,\\
0&\text{if }f^{(j)}(x)=0.
\end{cases}
\end{align*}
For any $m\in\mathbb{N}$ and $j\in\{1,\dots,d\}$,
let $\vec{g}_{j,m}:=\sigma_j \vec{f}_m\in L_{\rm c}^\infty$.
Since $\sigma_j$ is scalar-valued, it follows that
\begin{align*}
\left|W\vec{g}_{j,m}\right|=
|\sigma_j|\,\left|W\vec{f}_m\right|
\leq\left|W\vec{f}\right|.
\end{align*}
Thus, \eqref{eq:bounded} yields
\begin{align*}
\|I_\alpha \vec{g}_{j,m}\|_{L^q(W^q)}\lesssim
[W]_{\mathscr A_{p,q}}^{(1-\frac{\alpha}{n})\max\{1,\frac{p'}{q}\}}
\left\|\vec{f}\right\|_{L^p(W^p)},
\end{align*}
where the implicit positive constant is independent of
$j$ and $m$. Observe that, for any $j\in\{1,\dots,d\}$,
$m\in\mathbb{N}$, and $x\in\mathbb{R}^n$,
\begin{align*}
A_{j,m}(x):=\int_{\mathbb{R}^n}
\frac{|f^{(j)}(y)|\mathbf{1}_{E_m}(y)}{|x-y|^{n-\alpha}}\,dy
=\left[I_\alpha \vec{g}_{j,m}\right]^{(j)}(x),
\end{align*}
where $[I_\alpha \vec{g}_{j,m}]^{(j)}$ denotes the
$j$-th component function of $I_\alpha \vec{g}_{j,m}$.
Let $e_j:=(0,\dots,0,1,0,\dots,0)^T\in \mathbb C^d$ be the
vector with 1 in the $j$-th entry and 0 elsewhere.
By the definition of matrix weights, we find that,
for almost every $x\in\mathbb{R}^n$, $W(x)$ is positive definite,
and hence, for such $x$,
\begin{align*}
A_{j,m}(x)\leq\left|W^{-1}(x)e_j\right|
\left|W(x)I_\alpha \vec{g}_{j,m}(x)\right|.
\end{align*}
Applying \eqref{eq:bounded}, we conclude that,
for any $j\in\{1,\dots,d\}$ and $m\in\mathbb{N}$,
\begin{align}\label{eq:coord-bound}
\int_{\mathbb{R}^n}A_{j,m}(x)^q\left|W^{-1}(x)e_j\right|^{-q}\,dx
\lesssim[W]_{\mathscr A_{p,q}}^{(1-\frac{\alpha}{n})
\max\{p',q\}}\left\|\vec{f}\right\|_{L^p(W^p)}^q.
\end{align}
Observe that, for any $x\in\mathbb{R}^n$, $A_{j,m}(x)$ increases pointwise to
\begin{align*}
A_j(x):=\int_{\mathbb{R}^n}\frac{|f^{(j)}(y)|}{|x-y|^{n-\alpha}}\,dy
\end{align*}
as $m\to\infty$. From the monotone convergence theorem
and \eqref{eq:coord-bound},
we infer that, for any $j\in\{1,\dots,d\}$,
\begin{align*}
\int_{\mathbb{R}^n}A_j(x)^q\left|W^{-1}(x)e_j\right|^{-q}\,dx<\infty.
\end{align*}
By this and the fact that $0<|W^{-1}(x)e_j|<\infty$ for almost every $x\in\mathbb{R}^n$,
we find that $A_j(x)<\infty$ for almost every $x\in\mathbb{R}^n$.
Summing over $j\in\{1,\dots,d\}$, we conclude that,
for almost every $x\in\mathbb R^n$,
\begin{align*}
\int_{\mathbb{R}^n}\frac{|\vec{f}(y)|}{|x-y|^{n-\alpha}}\,dy
\leq\sum_{j=1}^d A_j(x)<\infty,
\end{align*}
and hence \eqref{eq:absolute} holds.
From \eqref{eq:absolute} and
the Lebesgue dominated convergence theorem,
we deduce that, for almost every $x\in\mathbb{R}^n$,
\begin{align*}
I_\alpha \vec{f}_m(x)\to
\int_{\mathbb{R}^n}\frac{\vec{f}(y)}{|x-y|^{n-\alpha}}\,dy=I_\alpha \vec{f}(x)
\end{align*}
and hence
\begin{align*}
\left|W(x)I_\alpha \vec{f}_m(x)\right|
\to\left|W(x)I_\alpha \vec{f}(x)\right|
\end{align*}
as $m\to\infty$. Applying Fatou's lemma and \eqref{eq:bounded},
we conclude that \eqref{eq:main} holds.

Finally, we show the optimality of the exponent of the weight constant
in \eqref{eq:main}. Taking $W:=wI_d$
for $w\in A_{p,q}$ and a vector-valued function with only one non-zero
component, we find that estimate \eqref{eq:main}
reduces to the scalar estimate \eqref{eq:scalar}.
Thus, the optimality of the exponent of the weight constant in \eqref{eq:main} follows
from that of \eqref{eq:scalar}. This completes the proof of Theorem \ref{thm:main}.
\end{proof}

\begin{remark}
Note that the proof of Theorem \ref{thm:main}
strongly relies on the assumption $p<q$.
This is because, in this proof, we use
Theorem \ref{thm-chain-packing} with $r=p$ and $s=q'$.
The assumption of Theorem \ref{thm:main}
guarantees that $\Gamma:=\frac{1}{p}+\frac{1}{q'}
=1+\frac{\alpha}{n}\in(1,\infty)$, which satisfies the assumption of
Theorem \ref{thm-chain-packing}.
Moreover, when $p=q$, we instead have $\Gamma=1$.
Thus, the off-diagonal chain-packing principle of dyadic cubes
in Theorem \ref{thm-chain-packing} can no longer
be applied (see Remark \ref{rek-chain-endpoint}),
and the present proof of Theorem \ref{thm:main}
does not cover the diagonal case $p=q$.
In particular, the proof of Theorem \ref{thm:main}
cannot be used to establish matrix-weighted estimates
for Calder\'on--Zygmund operators.
\end{remark}

Finally, we note that the proof of Theorem \ref{thm:main} also holds
for the following fractional kernel operators.
For any $\alpha\in(0,n)$, a linear
operator $T_\alpha$ is called a \emph{fractional kernel operator}
if there exist a complex-valued measurable function $K_\alpha$ on
$\{(x,y)\in\mathbb{R}^n\times\mathbb{R}^n:\,x\neq y\}$
and a positive constant $C$ such that,
for any $x,y\in\mathbb{R}^n$ with $x\neq y$,
\begin{align}\label{eq:fsi-kernel-size}
|K_\alpha(x,y)|\leq\frac{C}{|x-y|^{n-\alpha}}
\end{align}
and, for any $\vec f\in L_{\rm c}^\infty$ and $x\in\mathbb{R}^n$,
\begin{align}\label{eq:fsi-kernel-representation}
T_\alpha\vec f(x)=\int_{\mathbb R^n}K_\alpha(x,y)\vec f(y)\,dy.
\end{align}
By examining the proof of Theorem \ref{thm:main},
we find that the same argument is also valid for
fractional kernel operators.
More precisely, we have the following corollary.

\begin{corollary}\label{cor:fsi}
Let $\alpha$, $p$, and $q$ satisfy \eqref{eq:range}, and let
$T_\alpha$ be a fractional kernel operator satisfying
\eqref{eq:fsi-kernel-representation} and \eqref{eq:fsi-kernel-size}.
Then, for any $W\in\mathscr A_{p,q}$ and
$\vec f\in L_{\rm c}^\infty$,
\begin{align*}
\|T_\alpha\vec f\|_{L^q(W^q)}\lesssim
[W]_{\mathscr A_{p,q}}^{(1-\frac{\alpha}{n})\max\{1,\frac{p'}{q}\}}
\|\vec f\|_{L^p(W^p)},
\end{align*}
where the implicit positive constant is independent of
$W$ and $\vec{f}$. Consequently, $T_\alpha$ extends uniquely
to a bounded linear operator from $L^p(W^p)$ to $L^q(W^q)$.
\end{corollary}

\section{Proof of Theorem \ref{thm:matrix-weighted-Sobolev}}\label{sec:sobolev}

In this section, we prove Theorem \ref{thm:matrix-weighted-Sobolev}.
To this end, we construct a sparse family and obtain a
sparse domination for $|W\vec f|$, whose coefficients involve
the local $L^p$-norms of $WD\vec f$ over pairwise disjoint sets.
Using this domination and the off-diagonal Carleson-type embedding
in Lemma \ref{lem:off-diagonal-Carleson}, we finally show
Theorem \ref{thm:matrix-weighted-Sobolev}.

We start with two basic lemmas. The following
properties of level sets follow from the continuity
of the Lebesgue measure; we omit the details.
\begin{lemma}\label{lem:Sobolev-quantile}
Let $Q$ be a cube in $\mathbb R^n$.
Suppose that $v$ is a non-negative
measurable function on $Q$ that is finite almost everywhere.
Let $\delta\in(0,1)$, and let
\begin{align*}
\gamma:=\inf\left\{\lambda\in[0,\infty):\,
\left|\{x\in Q:\,v(x)>\lambda\}\right|\leq \delta|Q|\right\}.
\end{align*}
Then
\begin{align*}
\left|\{x\in Q:\,v(x)>\gamma\}\right|
\le\delta|Q|\leq\left|\{x\in Q:\,v(x)\ge\gamma\}\right|.
\end{align*}
\end{lemma}

Suppose that $\mathcal K\subset\mathbb C^d$ is
a nonempty closed convex set. For any $\vec{z}\in\mathbb C^d$, let
\begin{align*}
d(\vec{z},\mathcal K):=\inf_{\vec{y}\in\mathcal K}|\vec{z}-\vec{y}|
\end{align*}
be the \emph{distance from $\vec{z}$ to $\mathcal{K}$}.

Given a continuously differentiable vector-valued function
and a nonempty closed convex set, the next lemma constructs
a continuously differentiable scalar
function whose values depend on the distance between
the values of the given vector-valued function and the convex set.

\begin{lemma}\label{lem:Sobolev-smooth-truncation}
Suppose that $\mathcal K$ is a nonempty closed convex subset of $\mathbb C^d$.
Let $\vec{g}$ be a continuously differentiable function
from $\mathbb R^n$ to $\mathbb C^d$, and let $\gamma\in(0,\infty)$.
Then there exists a real-valued continuously differentiable
function $\psi$ on $\mathbb R^n$ such that, for any $x\in\mathbb R^n$,
\begin{align}\label{eq:trunc-values}
\begin{cases}
\psi(x)=0
&\text{if }\vec{g}(x)\in \mathcal{K},\\
\psi(x)\in(0,\gamma)
&\text{if }d(\vec{g}(x),\mathcal{K})\in(0,\gamma),\\
\psi(x)=\gamma&\text{if }d(\vec{g}(x),\mathcal{K})\geq\gamma,
\end{cases}
\end{align}
and
\begin{align}\label{eq:trunc-grad}
|\nabla\psi(x)|\lesssim \|D\vec{g}(x)\|\mathbf{1}_{\Delta}(x),
\end{align}
where
\begin{align*}
\Delta:=\left\{x\in\mathbb R^n:\,d(\vec{g}(x),\mathcal{K})\in(0,\gamma)\right\}
\end{align*}
and the implicit positive constant is independent of
$\mathcal{K}$, $\vec{g}$, $\gamma$, $\psi$, and $x$.
\end{lemma}

\begin{proof}
We first prove this lemma when $\mathcal K$ is a nonempty closed convex subset
of $\mathbb R^d$ and $\vec{g}$ is a continuously differentiable function
from $\mathbb R^n$ to $\mathbb R^d$.
To this end, for any $\vec{z}\in \mathbb{R}^d$,
let $v(\vec{z}):=[d(\vec{z},\mathcal{K})]^2$.
By \cite[Theorem 1.5.5]{fp03}, we find that, for any $\vec{z}\in\mathbb R^d$,
there exists a unique point $P_{\mathcal K}(\vec{z})\in\mathcal K$ such that
$|\vec{z}-P_{\mathcal K}(\vec{z})|=d(\vec{z},\mathcal K)$, and that
$v$ is continuously differentiable on $\mathbb R^d$.
Moreover, for any $\vec{z}\in\mathbb R^d$,
\begin{align}\label{eq:dist-grad}
\nabla v(\vec{z})=2\left[\vec{z}-P_{\mathcal K}(\vec{z})\right].
\end{align}
For any $x\in\mathbb R^n$, let $\rho(x):=v(\vec{g}(x))$.
Using \eqref{eq:dist-grad} and the chain rule,
we conclude that, for any $x\in\mathbb R^n$,
\begin{align*}
\nabla\rho(x)
=2[D\vec{g}(x)]^T\left[\vec{g}(x)-P_{\mathcal K}(\vec{g}(x))\right],
\end{align*}
where $[D\vec{g}(x)]^T$ is the transpose matrix of $D\vec{g}(x)$.
Hence,
\begin{align}\label{eq:rho-grad}
|\nabla\rho(x)|&\leq2\|D\vec{g}(x)\|
\left|\vec{g}(x)-P_{\mathcal K}(\vec{g}(x))\right|\\
&=2\|D\vec{g}(x)\|d(\vec{g}(x),\mathcal{K}).\nonumber
\end{align}
Let $\eta$ be the function on $\mathbb{R}$ defined by setting,
for any $s\in\mathbb{R}$,
\begin{align*}
\eta(s):=
\begin{cases}
0&\text{if }s\le0,\\
3s^2-2s^3&\text{if }0<s<1,\\
1&\text{if }s\ge1.
\end{cases}
\end{align*}
Observe that $\eta$ is continuously differentiable.
Moreover, for any $x\in\mathbb R^n$, let
$\psi(x):=\gamma\eta(\frac{\rho(x)}{\gamma^2})$.
It follows from the definition of $\psi$ that
\eqref{eq:trunc-values} holds.
For any $x\in\mathbb R^n$, the chain rule and
\eqref{eq:rho-grad} yield
\begin{align}\label{eq:psi-deriv}
|\nabla\psi(x)|&\le2\left|\eta'\!\left(\frac{\rho(x)}{\gamma^2}\right)
\right|\frac{\|D\vec{g}(x)\|d(\vec{g}(x),\mathcal{K})}{\gamma}.
\end{align}
Note that, for any $s\in\mathbb R\setminus(0,1)$,
$\eta'(s)=0$. Thus, for any $x\in\mathbb{R}^n\setminus \Delta$,
$\frac{\rho(x)}{\gamma^2}=0$ or $\frac{\rho(x)}{\gamma^2}\geq1$,
and hence $|\nabla\psi(x)|=0$.
On the other hand, for any $x\in\Delta$,
$d(\vec g(x),\mathcal K)<\gamma$, which, together with \eqref{eq:psi-deriv}, implies that
\begin{align*}
|\nabla\psi(x)|\le2\sup_{t\in\mathbb{R}}|\eta'(t)|\,\|D\vec{g}(x)\|.
\end{align*}
Therefore, $\psi$ satisfies \eqref{eq:trunc-grad}.
This completes the proof of this lemma
when $\mathcal K$ is a nonempty closed convex subset
of $\mathbb R^d$ and $\vec{g}$ is a continuously differentiable function
from $\mathbb R^n$ to $\mathbb R^d$.

We finally prove this lemma in the complex case.
Suppose that $\mathcal K$ is a
nonempty closed convex subset of $\mathbb C^d$ and
$\vec{g}$ is a continuously differentiable function
from $\mathbb R^n$ to $\mathbb C^d$.
For any $\vec z:=(z_1,\dots,z_d)\in\mathbb C^d$, let
\begin{align*}
\Phi(\vec z):=
({\rm Re}\,z_1,\dots,{\rm Re}\,z_d,
{\rm Im}\,z_1,\dots,{\rm Im}\,z_d),
\end{align*}
where, for any $z\in\mathbb C$, ${\rm Re}\,z$ and ${\rm Im}\,z$
denote the real and the imaginary parts of $z$, respectively.
From the definition of $\Phi$, we deduce that
$\widetilde{\mathcal K}:=\{\Phi(\vec z):\vec z\in\mathcal K\}$
is a nonempty closed convex subset of $\mathbb R^{2d}$,
$\Phi\circ\vec g$ is a continuously differentiable function
from $\mathbb R^n$ to $\mathbb R^{2d}$, and,
for any $x\in\mathbb R^n$,
\begin{align*}
d(\Phi(\vec g(x)),\widetilde{\mathcal K})=
d(\vec g(x),\mathcal K).
\end{align*}
Applying the preceding argument to
$\widetilde{\mathcal K}$ and $\Phi\circ\vec g$,
we obtain a real-valued continuously differentiable
function $\psi$ on $\mathbb R^n$ satisfying \eqref{eq:trunc-values}.
To verify \eqref{eq:trunc-grad}, it remains to show that, for any
$x\in\mathbb R^n$,
\begin{align}\label{eq:realification-derivative}
\|D(\Phi\circ\vec g)(x)\|\leq
\|D\vec g(x)\|.
\end{align}
Since $D(\Phi\circ\vec g)(x)$ is a real matrix for
any $x\in\mathbb{R}^n$, it follows that,
for any $\xi:=\xi_1+i\xi_2\in\mathbb C^n$ with
$\xi_1,\xi_2\in\mathbb R^n$,
\begin{align*}
|D(\Phi\circ\vec g)(x)\xi|^2
&=|D(\Phi\circ\vec g)(x)\xi_1|^2
+|D(\Phi\circ\vec g)(x)\xi_2|^2\\
&=|D\vec g(x)\xi_1|^2+|D\vec g(x)\xi_2|^2\\
&\leq\|D\vec g(x)\|^2\left(|\xi_1|^2+|\xi_2|^2\right)
=\|D\vec g(x)\|^2|\xi|^2,
\end{align*}
and hence \eqref{eq:realification-derivative} holds.
Therefore, we obtain the desired function $\psi$
in the complex case. This completes the proof of
Lemma \ref{lem:Sobolev-smooth-truncation}.
\end{proof}

We now prove Theorem \ref{thm:matrix-weighted-Sobolev}.

\begin{proof}[Proof of Theorem \ref{thm:matrix-weighted-Sobolev}]
Let $\vec{f}\in C_{\rm c}^\infty$.
We begin with the stopping-time construction.
By the assumption that $\vec{f}$ has compact support
and Lemma \ref{lem:shifted-dyadic}(ii), we find that
there exist a shifted dyadic
grid $\mathcal{D}$ and $Q_0\in\mathcal{D}$ such that
\begin{align}\label{eq:root-size}
\operatorname{supp}(\vec f)\subset Q_0
\text{ and }\left|\operatorname{supp}(\vec f)\right|
\leq\frac{1}{2}|Q_0|.
\end{align}
Let $\mathcal K_{Q_0}:=\{\mathbf{0}\}$.
From \eqref{eq:root-size}, we deduce that
\begin{align*}
\left|\left\{x\in Q_0:\,\vec f(x)\in\mathcal K_{Q_0}
\right\}\right|\geq\frac{1}{2}|Q_0|.
\end{align*}
Suppose that $Q$ is a stopping cube and that a nonempty
compact convex set $\mathcal K_Q\subset\mathbb C^d$ has been
constructed such that
\begin{align}\label{eq:K-good}
\left|\left\{x\in Q:\,
\vec f(x)\in\mathcal K_Q
\right\}\right|\geq\frac12|Q|.
\end{align}
For any cube $Q\subset\mathbb R^n$, let $A_Q$ be a reducing
operator of order $q$ for $W$.
For any $x\in Q$, let
$v_Q(x):=d(A_Q\vec f(x),A_Q\mathcal K_Q)$.
Let $\delta:=\frac{1}{2^{n+3}}$, and let
\begin{align*}
\gamma_Q:=\inf\left\{\lambda\in[0,\infty):\,
\left|\left\{x\in Q:\,v_Q(x)>\lambda\right\}\right|
\leq\delta|Q|\right\}.
\end{align*}
Since $\vec f\in C_{\rm c}^\infty$, it follows that
$v_Q$ is bounded on $Q$, and hence $\gamma_Q<\infty$.
By Lemma \ref{lem:Sobolev-quantile}, we find that
\begin{align}\label{eq:level-upper}
\left|\left\{x\in Q:\,v_Q(x)>\gamma_Q
\right\}\right|\leq\delta|Q|
\end{align}
and
\begin{align}\label{eq:level-lower}
\left|\left\{x\in Q:\,v_Q(x)\geq\gamma_Q
\right\}\right|\geq\delta|Q|.
\end{align}
Let
\begin{align*}
H_Q:=\left\{x\in Q:\,v_Q(x)>\gamma_Q\right\}.
\end{align*}
Let $\operatorname{ch}(Q)$ be the collection of maximal dyadic
subcubes $P\subsetneqq Q$ such that $|P\cap H_Q|>2\delta|P|$.
Using the Lebesgue differentiation theorem, we conclude that
$\bigcup_{P\in\operatorname{ch}(Q)}P$ contains almost every
point of $H_Q$. Moreover, by \eqref{eq:level-upper}
and the disjointness of cubes in
$\operatorname{ch}(Q)$, we find that
\begin{align}\label{eq:child-sum}
\sum_{P\in\operatorname{ch}(Q)}|P|
\leq\frac{1}{2\delta}|H_Q|
\leq\frac{1}{2}|Q|.
\end{align}
If $P\in\operatorname{ch}(Q)$ and $\widehat P$ is its dyadic
parent, then $|\widehat P\cap H_Q|\leq2\delta|\widehat P|$.
Indeed, this inequality follows from the maximality of $P$ if
$\widehat P\subsetneqq Q$, and from \eqref{eq:level-upper}
if $\widehat P=Q$. Hence,
\begin{align}\label{eq:child-good}
\left|P\cap H_Q\right|\leq
\left|\widehat P\cap H_Q\right|
\leq2^{n+1}\delta|P|=\frac{1}{4}|P|.
\end{align}
For any $P\in\operatorname{ch}(Q)$, let
\begin{align}\label{eq:K-child}
\mathcal K_P:=\mathcal K_Q+
\gamma_Q A_Q^{-1}\overline{B(\mathbf{0},1)},
\end{align}
where $\overline{B(\mathbf{0},1)}:=
\{\vec{z}\in\mathbb{C}^d:\,|\vec{z}|\leq1\}$
and the sum denotes the usual Minkowski sum.
Then $\mathcal K_P$ is again a nonempty compact convex subset of
$\mathbb C^d$. If $x\in P\setminus H_Q$, then
$v_Q(x)\leq\gamma_Q$, and hence
\begin{align*}
A_Q\vec f(x)\in A_Q\mathcal K_Q
+\gamma_Q\overline{B(\mathbf{0},1)}.
\end{align*}
Equivalently, $\vec f(x)\in\mathcal K_P$. Thus,
applying this and \eqref{eq:child-good}, we obtain
\begin{align*}
\left|\left\{x\in P:\,
\vec f(x)\in\mathcal K_P\right\}
\right|\geq\left|P\setminus H_Q\right|
\geq\frac{3}{4}|P|,
\end{align*}
which verifies \eqref{eq:K-good} for $P$.

Iterating this construction with $Q$ replaced by
cubes in $\operatorname{ch}(Q)$, we obtain a stopping family $\mathcal{S}$
with root $Q_0$, the nonempty compact convex sets
$\{\mathcal K_Q\}_{Q\in\mathcal{S}}$, and the truncation levels
$\{\gamma_Q\}_{Q\in\mathcal{S}}$. For any $Q\in\mathcal{S}$, let
\begin{align*}
E_Q:=Q\setminus
\bigcup_{P\in\operatorname{ch}(Q)}P.
\end{align*}
Observe that the sets
$\{E_Q\}_{Q\in\mathcal{S}}$ are pairwise disjoint.
Moreover, by \eqref{eq:child-sum},
for any $Q\in\mathcal{S}$, $|E_Q|\geq\frac{1}{2}|Q|$.
Thus, $\mathcal{S}$ is $\frac{1}{2}$-sparse.

We estimate each truncation level $\gamma_Q$ by the
corresponding local norm of the weighted gradient. For any $Q\in\mathcal{S}$, let
$\Delta_Q:=\{x\in Q:v_Q(x)\in(0,\gamma_Q)\}$
($\Delta_Q:=\emptyset$ if $\gamma_Q=0$).
We claim that, for any $Q\in\mathcal{S}$,
\begin{align}\label{eq:level-bound}
\gamma_Q\lesssim[W]_{\mathscr A_{p,q}}^{\frac{1}{q}}|Q|^{-\frac{1}{q}}
\left\|WD\vec f\,\mathbf{1}_{\Delta_Q}\right\|_{L^p(\mathbb R^n,\mathbb C^{d\times n})}.
\end{align}
If $\gamma_Q=0$, then $\Delta_Q=\emptyset$ and
\eqref{eq:level-bound} is immediate. Thus, it suffices to
prove \eqref{eq:level-bound} for $Q\in\mathcal{S}$
satisfying $\gamma_Q\in(0,\infty)$.
Applying Lemma \ref{lem:Sobolev-smooth-truncation} with
$\vec g:=A_Q\vec f$, $\mathcal K:=A_Q\mathcal K_Q$, and $\gamma:=\gamma_Q$,
we obtain a real-valued continuously differentiable
function $\psi_Q$ on $\mathbb R^n$ satisfying
$0\leq\psi_Q\leq\gamma_Q$,
\begin{align}\label{eq:psi-zero}
\psi_Q=0\text{ on }
\left\{x\in Q:\,A_Q\vec{f}(x)\in A_Q\mathcal K_Q\right\},
\end{align}
\begin{align}\label{eq:psi-level}
\psi_Q=\gamma_Q\text{ on }
\left\{x\in Q:\,v_Q(x)\geq\gamma_Q\right\},
\end{align}
and, for any $x\in Q$,
\begin{align}\label{eq:psi-grad}
|\nabla\psi_Q(x)|\lesssim
\left\|A_QD\vec f(x)\right\|\mathbf{1}_{\Delta_Q}(x).
\end{align}
By \eqref{eq:K-good}, the measure of the set
in \eqref{eq:psi-zero} is at least
$\frac{1}{2}|Q|$. From \eqref{eq:level-lower}, we infer that
the measure of the set
in \eqref{eq:psi-level} is
at least $\delta|Q|$. Therefore,
\begin{align}\label{eq:osc-lower}
\gamma_Q|Q|
\lesssim
\|\psi_Q-(\psi_Q)_Q\|_{L^1(Q)},
\end{align}
where $(\psi_Q)_Q=\fint_Q \psi_Q(x)\,dx$.
Indeed, if
$(\psi_Q)_Q\leq\frac{\gamma_Q}{2}$, then
$\frac{\gamma_Q}{2}\leq|\psi_Q-(\psi_Q)_Q|$
on the set where $\psi_Q=\gamma_Q$, and hence \eqref{eq:osc-lower}
holds in this case. Otherwise, the same lower
bound holds on the set where $\psi_Q=0$, which further implies
\eqref{eq:osc-lower} in this case.
Applying the Poincar\'e inequality
(see, for instance, \cite[(7.45)]{gt01}), we obtain
\begin{align*}
\|\psi_Q-(\psi_Q)_Q\|_{L^1(Q)}
\lesssim
\ell(Q)\int_Q |\nabla\psi_Q(x)|\,dx.
\end{align*}
Combining this with \eqref{eq:psi-grad} and
\eqref{eq:osc-lower}, we conclude that
\begin{align}\label{eq:grad-L1}
\gamma_Q|Q|^{\frac{n-1}{n}}
\lesssim
\int_{\Delta_Q}
\left\|A_QD\vec f(x)\right\|\,dx.
\end{align}
Let $G:=WD\vec f$.
Observe that, for almost every $x\in Q$,
\begin{align*}
\left\|A_Q D\vec f(x)\right\|=
\left\|A_Q W^{-1}(x) G(x)\right\|
\leq\left\|A_QW^{-1}(x)\right\|\left\|G(x)\right\|
=: b_Q(x)\left\|G(x)\right\|.
\end{align*}
Applying H\"older's inequality together with
\eqref{eq:dual-reduce} when $p\in(1,n)$ and
\eqref{eq:end-reduce} when $p=1$, we obtain
\begin{align}\label{eq:grad-weight}
\int_{\Delta_Q}\left\|A_Q D\vec f(x)\right\|\,dx
&\leq\|b_Q\|_{L^{p'}(Q)}
\|G\mathbf{1}_{\Delta_Q}\|_{L^p(\mathbb R^n,\mathbb C^{d\times n})}\\
&\lesssim[W]_{\mathscr A_{p,q}}^{\frac{1}{q}}|Q|^{\frac{1}{p'}}
\|G\mathbf{1}_{\Delta_Q}\|_{L^p(\mathbb R^n,\mathbb C^{d\times n})}\nonumber
\end{align}
Note that $\frac{1}{p'}-\frac{n-1}{n}=-\frac{1}{q}$.
Combining \eqref{eq:grad-L1} and
\eqref{eq:grad-weight}, we conclude that
the claim \eqref{eq:level-bound} holds.

We now prove that the sets
$\{\Delta_Q\}_{Q\in\mathcal{S}}$ are pairwise disjoint. Since two
stopping cubes are either disjoint or nested, it suffices to show
$\Delta_Q\cap \Delta_R=\emptyset$ for $Q,R\in\mathcal{S}$ with
$R\subsetneqq Q$. Let $P$ be the stopping child of $Q$
that contains $R$. If $x\in\Delta_Q\cap R$, then
$v_Q(x)<\gamma_Q$, and hence
\begin{align*}
\vec f(x)\in\mathcal K_Q+
\gamma_Q A_Q^{-1}\overline{B(\mathbf{0},1)}
=\mathcal K_P.
\end{align*}
Thus, $v_P(x)=0$. Moreover, \eqref{eq:K-child} yields
$\mathcal K_Q\subset\mathcal K_P$ for every stopping child $P$
of $Q$. Iterating this inclusion along the stopping chain from $P$
to $R$, we obtain $\mathcal K_P\subset\mathcal K_R$. Therefore,
$\vec f(x)\in\mathcal K_R$, so $v_R(x)=0$ and
$x\notin\Delta_R$. Hence the sets
$\{\Delta_Q\}_{Q\in\mathcal{S}}$ are pairwise disjoint.
For any $Q\in\mathcal{S}$, let $\lambda_Q
:=\|G\mathbf{1}_{\Delta_Q}\|_{L^p(\mathbb R^n,\mathbb C^{d\times n})}$.
Then
\begin{align}\label{eq:grad-sum}
\sum_{Q\in\mathcal{S}}\lambda^p_Q=\sum_{Q\in\mathcal{S}}
\|G\mathbf{1}_{\Delta_Q}\|_{L^p(\mathbb R^n,\mathbb C^{d\times n})}^p
\leq\|G\|_{L^p(\mathbb R^n,\mathbb C^{d\times n})}^p.
\end{align}

For any $Q\in\mathcal{S}$ and $x\in \mathbb{R}^n$, let
$a_Q(x):=\|W(x)A_Q^{-1}\|$.
We next prove that,
for almost every $x\in \mathbb{R}^n$,
\begin{align}\label{eq:pointwise}
\left|W(x)\vec f(x)\right|\leq
\sum_{R\in\mathcal{S}}\gamma_Ra_R(x)\mathbf{1}_R(x).
\end{align}
From the definitions of $H_Q$ and $v_Q$, we deduce that,
for any $Q\in\mathcal{S}$ and for almost every $x\in E_Q$,
$x\notin H_Q$, and hence $v_Q(x)\leq\gamma_Q$. Thus,
$\vec f(x)\in\mathcal K_Q
+\gamma_Q A_Q^{-1}\overline{B(\mathbf{0},1)}$.
Iterating \eqref{eq:K-child} along the stopping
ancestors of $Q$, we obtain
\begin{align}\label{eq:K-sum}
\mathcal K_Q
+\gamma_Q A_Q^{-1}\overline{B(\mathbf{0},1)}
=\sum_{R\in\mathcal{S}, Q\subset R}
\gamma_R A_R^{-1}\overline{B(\mathbf{0},1)},
\end{align}
where the sum is a finite Minkowski sum. Applying
\eqref{eq:K-sum}, we conclude that there exist vectors
$\{\vec z_R\}_{R\in\mathcal{S}, Q\subset R}$ such that
$|\vec z_R|\leq\gamma_R$ and
\begin{align}\label{eq-f-sum}
\vec f(x)=\sum_{R\in\mathcal{S}, Q\subset R}
A_R^{-1}\vec z_R.
\end{align}
Multiplying both sides of \eqref{eq-f-sum}
by $W(x)$ and taking the vector norm,
we find that, for almost every $x\in E_Q$,
\begin{align}\label{eq:pointwise-E}
\left|W(x)\vec f(x)\right|\leq
\sum_{R\in\mathcal{S}, Q\subset R}
\gamma_R\left\|W(x)A_R^{-1}\right\|
=\sum_{R\in\mathcal{S}, Q\subset R}\gamma_R a_R(x).
\end{align}
By \eqref{eq:child-sum}, we find that,
for any $m\in\mathbb{N}$, the total measure of the
$m$-th stopping generation is at most $2^{-m}|Q_0|$.
Hence almost every point of $Q_0$ belongs to $E_Q$ for some
$Q\in\mathcal{S}$. Consequently, it follows from \eqref{eq:pointwise-E}
that \eqref{eq:pointwise} holds for
almost every $x\in Q_0$. Moreover,
\eqref{eq:pointwise} obviously holds for
$x\in \mathbb{R}^n\setminus Q_0$ because
$\vec f=\mathbf{0}$ outside $Q_0$.
This completes the proof of \eqref{eq:pointwise}.

Combining \eqref{eq:level-bound} and
\eqref{eq:pointwise}, we conclude that,
for almost every $x\in\mathbb R^n$,
\begin{align*}
|W(x)\vec f(x)|
\lesssim
[W]_{\mathscr A_{p,q}}^{\frac{1}{q}}
\sum_{Q\in\mathcal{S}}
|Q|^{-\frac{1}{q}}
\lambda_Qa_Q(x)\mathbf{1}_Q(x).
\end{align*}
To prove \eqref{eq:sobolev}, let $g\in L^{q'}$ be non-negative.
Applying Tonelli's theorem, we obtain

\begin{align*}
\int_{\mathbb R^n}\left|W(x)\vec f(x)\right|g(x)\,dx
\lesssim[W]_{\mathscr A_{p,q}}^{\frac{1}{q}}\sum_{Q\in\mathcal{S}}\lambda_Q
|Q|^{-\frac1q}\int_Qa_Q(x)g(x)\,dx.
\end{align*}
If $p\in(1,n)$, then, from H\"older's inequality
and Lemma \ref{lem:off-diagonal-Carleson} with
$\alpha:=1$ and $\eta:=\frac12$, we deduce that
\begin{align}\label{eq:sob-p}
\int_{\mathbb R^n}\left|W(x)\vec f(x)\right|g(x)\,dx
&\lesssim[W]_{\mathscr A_{p,q}}^{\frac{1}{q}}
\left(\sum_{Q\in\mathcal{S}}\lambda_Q^p\right)^{\frac1p}
\left\{\sum_{Q\in\mathcal{S}}\left[
|Q|^{-\frac1q}\int_Qa_Q(x)g(x)\,dx
\right]^{p'}\right\}^{\frac1{p'}}\\
&\lesssim[W]_{\mathscr A_{p,q}}^{\frac{1}{q}+\frac{1}{p'}}
\left(\sum_{Q\in\mathcal{S}}\lambda_Q^p\right)^{\frac1p}
\|g\|_{L^{q'}}.\nonumber
\end{align}
If $p=1$, by H\"older's inequality and \eqref{eq:reduce},
we find that, for any $Q\in\mathcal{S}$,
\begin{align*}
|Q|^{-\frac1q}\int_Qa_Q(x)g(x)\,dx
\leq\left(\fint_Q[a_Q(x)]^q\,dx\right)^{\frac1q}
\|g\mathbf{1}_Q\|_{L^{q'}}\lesssim\|g\|_{L^{q'}},
\end{align*}
and hence
\begin{align}\label{eq:sob-one}
\int_{\mathbb R^n}\left|W(x)\vec f(x)\right|g(x)\,dx
\lesssim[W]_{\mathscr A_{1,q}}^{\frac{1}{q}}
\sum_{Q\in\mathcal{S}}\lambda_Q\|g\|_{L^{q'}}.
\end{align}

Combining \eqref{eq:grad-sum} with \eqref{eq:sob-p}
when $p\in(1,n)$ and with \eqref{eq:sob-one} when $p=1$,
and taking the supremum over
all non-negative $g\in L^{q'}$ with
$\|g\|_{L^{q'}}\leq1$, we conclude that
\begin{align*}
\left\|W\vec f\right\|_{L^q}
&\lesssim[W]_{\mathscr A_{p,q}}^{\frac{1}{q}+\frac{1}{p'}}
\left(\sum_{Q\in\mathcal{S}}\lambda_Q^p
\right)^{\frac{1}{p}}\\
&\lesssim[W]_{\mathscr A_{p,q}}^{\frac{1}{q}+\frac{1}{p'}}
\|G\|_{L^p(\mathbb R^n,\mathbb C^{d\times n})}=[W]_{\mathscr A_{p,q}}^{\frac{n-1}{n}}
\left\|WD\vec f\right\|_{L^p(\mathbb R^n,\mathbb C^{d\times n})},
\end{align*}
where the last equality follows from the identity
$\frac{1}{q}+\frac{1}{p'}=\frac{n-1}{n}$. Thus,
\eqref{eq:sobolev} holds.

It remains to prove the optimality of the exponent
$\frac{n-1}{n}$ in \eqref{eq:sobolev}.
Let $w\in A_{p,q}$, and let $W:=wI_d$. By \eqref{eq-Apq},
\eqref{eq:Apq}, \eqref{eq:A1q-scalar}, and \eqref{eq:A1q}, we find
that, for any $p\in[1,n)$,
$[W]_{\mathscr A_{p,q}}=[w]_{A_{p,q}}$.
Thus, taking a vector-valued function with only one non-zero component
in \eqref{eq:sobolev} yields \eqref{eq:scalar-sobolev-intro},
which is the corresponding scalar-weighted
version of \eqref{eq:sobolev}.
Therefore, it remains to show that the exponent
$\frac{n-1}{n}$ of the weight constant $[w]_{A_{p,q}}$
in \eqref{eq:scalar-sobolev-intro} is optimal.
Note that the optimality of \eqref{eq:scalar-sobolev-intro}
with Lipschitz functions having compact support
has been proved in \cite[Subsection 8.4]{cm12} using
functions $f_\delta(x)=e^{-|x|^\delta}$ and power weights
$w_\delta(x):=|x|^{\frac{\delta-n}{q}}$
for $\delta\in(0,1)$ and $x\in\mathbb{R}^n$. Noting that $f_\delta\notin
C_{\rm c}^\infty$ and our inequality \eqref{eq:sobolev}
is established only for functions in $C_{\rm c}^\infty$,
instead of $f_\delta$ we use the following compactly
supported smooth test functions to show the optimality of the exponent
$\frac{n-1}{n}$ in \eqref{eq:scalar-sobolev-intro}.
Suppose that $\phi$ is a non-zero smooth function on $\mathbb R$
supported in $(1,2)$.
For any $\delta\in(0,1)$ and $x\in\mathbb{R}^n$, let
\begin{align*}
w_\delta(x):=|x|^{\frac{\delta-n}{q}}
\text{ and }f_\delta(x):=\phi(|x|^\delta).
\end{align*}
Then $f_\delta$ is a smooth and compactly supported
function on $\mathbb{R}^n$.
A standard computation for power weights yields
$[w_\delta]_{A_{p,q}}\sim\delta^{-1}$,
where the positive equivalence constants are independent of $\delta$.
Using polar coordinates and
the change of variables $t=r^\delta$, we obtain
\begin{align*}
\|w_\delta f_\delta\|_{L^q}^q
=\frac{\omega_{n-1}}{\delta}
\int_1^2|\phi(t)|^q\,dt
\end{align*}
and
\begin{align*}
\|w_\delta\nabla f_\delta\|_{L^p}^p=
\omega_{n-1}\delta^{p-1}\int_1^2
|\phi'(t)|^pt^{p(1+\frac{1}{q})-1}\,dt,
\end{align*}
where $\omega_{n-1}$ denotes the surface measure of the unit
sphere in $\mathbb R^n$. Hence, $\|w_\delta f_\delta\|_{L^q}
\sim\delta^{-\frac{1}{q}}$ and
$\|w_\delta\nabla f_\delta\|_{L^p}\sim\delta^{\frac{1}{p'}}$,
where all positive equivalence constants are independent of $\delta$.
Note that $\frac{1}{q}+\frac{1}{p'}=\frac{n-1}{n}$.
If the exponents of the weight constants in \eqref{eq:sobolev}
and hence \eqref{eq:scalar-sobolev-intro} can
be replaced by $\beta<\frac{n-1}{n}$, then, for any
$\delta\in(0,1)$, $1\lesssim \delta^{\frac{n-1}{n}-\beta}$,
where the implicit positive constant is independent of $\delta$.
This is impossible for sufficiently small $\delta$.
Thus, the exponent of the weight constant
in \eqref{eq:sobolev} is optimal.
This completes the proof of Theorem
\ref{thm:matrix-weighted-Sobolev}.
\end{proof}

\begin{remark}\label{rmk-sobolev}
Let $n\geq2$ be an integer, $p\in(1,n)$, $q:=\frac{np}{n-p}$,
and $W\in \mathscr A_{p,q}$.
We finally compare Theorem \ref{thm:matrix-weighted-Sobolev}
with \cite[Theorem 1.5]{clm26}. Recall that, for any $j\in\{1,\dots,n\}$,
the operator $I_1^j$ is defined by setting,
for any $\vec{g}\in L_{\rm c}^\infty$
and $x\in\mathbb{R}^n$,
\begin{align*}
I_1^j \vec{g}(x):=\frac{1}{\omega_{n-1}}\int_{\mathbb R^n}
\frac{x_j-y_j}{|x-y|^n}\vec{g}(y)\,dy.
\end{align*}
In \cite[Theorem 1.5]{clm26}, for
any $\vec f\in C_{\rm c}^\infty$,
using the representation
\begin{align*}
\vec f=\sum_{j=1}^n I_1^j\left(\partial_j\vec f\right)
\end{align*}
and the matrix-weighted estimate for $I_1^j$
established in \cite[Theorem 4.9]{clm26},
Cruz-Uribe et al. established a matrix-weighted Sobolev inequality
as in \eqref{eq:sobolev} with the exponent $\frac{n-1}{n}$
of the weight constant replaced by
$\min\{\frac{(n-1)p'}{nq}+\frac{1}{q'},\,
\frac{n-1}{n}+\frac{p'-1}{q}\}$. Note that,
for any $j\in\{1,\dots,n\}$, $I_1^j$ satisfies
\eqref{eq:fsi-kernel-representation} and \eqref{eq:fsi-kernel-size}
with $\alpha=1$. It follows from Corollary \ref{cor:fsi} that,
for any $j\in\{1,\dots,n\}$ and $\vec g\in C_{\rm c}^\infty$,
\begin{align}\label{eq-bound-Ij}
\left\|WI_1^j\vec g\right\|_{L^q}
\lesssim[W]_{\mathscr A_{p,q}}^{
\frac{n-1}{n}\max\{1,\frac{p'}{q}\}}
\|W\vec g\|_{L^p}.
\end{align}
Although this estimate is better than the one obtained by applying
\cite[Theorem 4.9]{clm26} to $I_1^j$ for any $j\in\{1,\dots,n\}$,
for any $\vec{f}\in C_{\rm c}^\infty$,
repeating the argument in the proof of \cite[Theorem 1.5]{clm26}
with this improved estimate only yields
\begin{align}\label{eq:sob-repr}
\left\|W\vec f\right\|_{L^q}
\lesssim[W]_{\mathscr A_{p,q}}^{
\frac{n-1}{n}\max\{1,\frac{p'}{q}\}}
\left\|WD\vec f\right\|_{L^p(\mathbb R^n,\mathbb C^{d\times n})}.
\end{align}
Since, when $p\in[\frac{2n}{n+1},n)$,
$q\geq p'$, \eqref{eq:sob-repr} in this case
provides the optimal exponent $\frac{n-1}{n}$ as in \eqref{eq:sobolev}.
However, when $p\in(1,\frac{2n}{n+1})$,
$p'>q$, and hence \eqref{eq:sob-repr} in this case
provides a larger exponent $\frac{(n-1)p'}{nq}$ than the one in
\eqref{eq:sobolev}.
Thus, the argument based on \eqref{eq-bound-Ij}
does not yield the sharp matrix-weighted Sobolev
estimate for the whole range $p\in(1,n)$ (a similar phenomenon
already occurs in the scalar setting; see \cite[Remark 2.8]{lmpt10}).
Moreover, this argument cannot be used when $p=1$ because of
the range restriction of Corollary \ref{cor:fsi}.
This explains why the proof of Theorem
\ref{thm:matrix-weighted-Sobolev} uses a different
stopping-time argument.
\end{remark}


\smallskip

\noindent\textbf{Acknowledgements}\quad The authors acknowledge the use of AI tools during the exploratory stage of this project.
All mathematical arguments and proofs in the final manuscript were checked
and written by the authors. The third author would like to thank Professor
Tuomas Hyt\"onen for valuable suggestions on improving the presentation of this article.

\smallskip

\noindent\textbf{Funding}\quad This project is partially supported by the National Natural
Science Foundation of China (Grant
Nos. 12371093 and 12431006), the Beijing Natural
Science Foundation (Grant No. 1262011), and
the Fundamental Research Funds for the Central Universities
(Grant No. 2253200028).

\bigskip

\noindent Dachun Yang, Wen Yuan and Mingdong Zhang.

\medskip

\noindent Laboratory of Mathematics and Complex Systems
(Ministry of Education of China),
School of Mathematical Sciences,
Institute for Advanced Study, Beijing Normal University,
Beijing 100875, The People's Republic of China

\smallskip

\noindent{{\it E-mails:}}
\texttt{dcyang@bnu.edu.cn} (D. Yang)

\noindent\phantom{{\it E-mails:}}
\texttt{wenyuan@bnu.edu.cn} (W. Yuan)

\noindent\phantom{\it E-mails:}
\texttt{mdzhang@mail.bnu.edu.cn} (M. Zhang)
\end{document}